\documentclass[11pt]{article}
\usepackage{microtype}
\usepackage{tikz}
\usepackage{latexsym,amsfonts,amsthm,amsmath,amscd,amssymb,color,url,bm}
\usepackage[a4paper,left=1.5in,right=1.5in]{geometry}
\usepackage{graphicx}
\usepackage{etoolbox}
\usepackage{xparse}
\usepackage{tabularray}

\newtheorem{lemma}{Lemma}[section]
\newtheorem{thm}[lemma]{Theorem}
\newtheorem{rem}[lemma]{Remark}
\newtheorem{prop}[lemma]{Proposition}

\newtheorem{conj}[lemma]{Conjecture}

\newtheorem{defn}[lemma]{Definition}

\newcommand\matP{{\mathbb{P}}}
\newcommand\matR{{\mathbb{R}}}

\newcommand\matN{{\mathbb{N}}}

\newcommand\matC{{\mathbb{C}}}

\renewcommand{\hbar}{{\overline{h}}}

\newfont{\Got}{eufm10 scaled 1200}
\newcommand{\permu}{{\hbox{\Got S}}}

\newcommand{\mycap} [1] {\caption{\footnotesize{#1}}}

\newcommand\calC{{\mathcal C}}

\newcommand\calD{{\mathcal D}}

\newcommand{\partition}[1]{{[\![#1]\!]}}

\ExplSyntaxOn
\NewDocumentCommand{\faifig}{mmm}{
    \begin{center}
    \tl_if_empty:nTF {#2}
        {\textbf{\tiny #1}}
        {
        \tikzsetnextfilename{#2}
        \include{Figs/#2}
        }
    \mycap{#3}
    \end{center}
}
\ExplSyntaxOff

\NewDocumentCommand{\matebold}{m}{\bm{#1}}

\usetikzlibrary{decorations.markings,decorations.pathreplacing,decorations.pathmorphing,math,calc,positioning,arrows.meta,fpu}
\usetikzlibrary{spath3}
\usepackage{tikz-3dplot}

\def\graphicfast{0}

\makeatletter
\def\createcontour@precision{0.01}

\newlength{\edgelinewidth}
\newlength{\contourradius}
\newlength{\edgehalo}
\tikzset{
disk 1 boundary/.style={densely dotted,black!60},
disk 2 boundary/.style={black!60},
disk 1 boundary enveloped/.style={disk 1 boundary,line width=1pt},
disk 2 boundary enveloped/.style={disk 2 boundary,line width=1pt},
disk 1 boundary dashed/.style={disk 1 boundary,opacity=.6}
}
\colorlet{disk 1}{black!20}
\colorlet{disk 2}{black!35}
\tikzset{colored label/.style={#1!60!black}}

\NewDocumentCommand{\tikzenumlabel}{m}{\tikz[baseline=0pt]\node[anchor=base,fill=black!60,text=white,text height={height("1")},font={\scriptsize\normalfont},minimum width={height("1")+depth("1")+3pt},rounded corners=3pt,inner sep=1.5pt] {\textbf{#1}};}

\tikzset{pics/arc from center and point/.style n args={3}{code={
\tikzmath{coordinate \c,\p,\d;\c=#1;\p=#2;\d=(\p)-(\c);\r=veclen(\dx,\dy);\a=atan2(\dy,\dx);}
\path[pic actions] (\p) arc(\a:{\a+#3}:{\r pt});
}}}

\tikzset{
my fancy arrow/.pic={\draw[black!70,line width=2pt,-{Latex[sharp,scale=.75]}] (-.33cm,0) -- (.33cm,0);},
my fancy double arrow/.pic={\draw[black!70,line width=2pt,{Latex[sharp,scale=.75]}-{Latex[sharp,scale=.75]}] (-.5cm,0) -- (.5cm,0);}
}

\pgfdeclarelayer{graph near vertex contour}
\pgfdeclarelayer{graph near vertex mask}
\pgfdeclarelayer{graph near vertex edge}
\pgfdeclarelayer{graph edge below}
\pgfdeclarelayer{graph edge above}
\pgfdeclarelayer{graph vertex}

\tikzset{
graph picture/.style={
/utils/exec={\pgfsetlayers{graph near vertex contour,graph near vertex mask,graph near vertex edge,graph edge below,graph edge above,graph vertex,main}}
}
}

\tikzset{
exec on layer/.code 2 args={\begin{pgfonlayer}{#1}#2\end{pgfonlayer}}
}

\newlength{\vertex@radius}
\NewDocumentCommand{\largevertices}{}{\setlength{\vertex@radius}{1.5pt}}
\NewDocumentCommand{\smallvertices}{}{\setlength{\vertex@radius}{1.5pt}}

\tikzset{
black vertex/.pic={\begin{pgfonlayer}{graph vertex}\filldraw[white,fill=black] circle(\vertex@radius);\end{pgfonlayer}},
white vertex/.pic={\begin{pgfonlayer}{graph vertex}\filldraw[black,fill=white] circle(\vertex@radius);\end{pgfonlayer}},
black edge/.style={line width=\edgelinewidth,draw=black!80},
thick black edge/.style={black edge,line width =1.5pt},
black edge dashed/.style={black edge,densely dashed},
red edge/.style={/pgf/fpu/install only={reciprocal},line width={.7*\edgelinewidth},draw=black!80,decorate,decoration={coil,segment length=1mm,post length=1mm,pre length=1mm,amplitude=.5mm}},
red edge dashed/.style={red edge,densely dashed},
green edge/.style={line width=\edgelinewidth,draw=green},
edge on layer/.style 2 args={postaction={decorate,decoration={show path construction,
	lineto code={\begin{pgfonlayer}{#1}\draw[tmpstyle/.style/.expand once={#2},tmpstyle] (\tikzinputsegmentfirst) -- (\tikzinputsegmentlast);\end{pgfonlayer}},
	curveto code={\begin{pgfonlayer}{#1}\draw[tmpstyle/.style/.expand once={#2},tmpstyle] (\tikzinputsegmentfirst) .. controls (\tikzinputsegmentsupporta) and (\tikzinputsegmentsupportb) .. (\tikzinputsegmentlast);\end{pgfonlayer}}}}},
graph edge/.style 2 args={edge on layer={graph edge #1}{#2}},
}

\tikzset{
create contour settings/.is family,
create contour settings/.cd,
defaults/.style={save in=contour},
direction/.store in=\createcontour@direction,
radius/.store in=\createcontour@radius,
begin/.store in=\createcontour@begin,
end/.store in=\createcontour@end,
save in/.store in=\createcontour@savein,
}
\tikzset{
create contour/.style={
	create contour settings/.cd,
	defaults,
	#1,
	/tikz/.cd,
	postaction={
	/utils/exec={
		\expandafter\xdef\csname\createcontour@savein\endcsname{}
		\expandafter\xdef\csname\createcontour@savein rev\endcsname{}
		\def\todoonmark{
			\edef\coordname{\createcontour@savein-\pgfkeysvalueof{/pgf/decoration/mark info/sequence number}}
			\expandafter\xdef\csname\createcontour@savein\endcsname{\csname\createcontour@savein\endcsname (\coordname)}
\expandafter\xdef\csname\createcontour@savein rev\endcsname{(\coordname-rev) \csname\createcontour@savein rev\endcsname}
			\coordinate (\coordname) at ({90*(\createcontour@direction)}:\createcontour@radius);
			\coordinate (\coordname-rev);
}},
	decorate,decoration={
		markings,
		mark=at position {\createcontour@begin} with {\coordinate (\createcontour@savein-begin-onedge);\coordinate (\createcontour@savein-begin) at ({90*(\createcontour@direction)}:\createcontour@radius);},
		mark=between positions {\createcontour@begin+0.0005} and {\createcontour@end-0.0005} step \createcontour@precision with {\todoonmark},
		mark=at position {\createcontour@end} with {\coordinate (\createcontour@savein-end-onedge);\coordinate (\createcontour@savein-end) at ({90*(\createcontour@direction)}:\createcontour@radius);\todoonmark}
	},
	}
}
}

\tikzset{
contour edge settings/.is family,
contour edge settings/.cd,
defaults/.style={below,begin=0.2,end=0.8,left bcap=0,right bcap=0,left ecap=0,right ecap=0,edge={opacity=0}},
right/.store in=\contouredge@right,
left/.store in=\contouredge@left,
below/.code={\def\contouredge@layer{below}\def\contouredge@isabove{0}},
above/.code={\def\contouredge@layer{above}\def\contouredge@isabove{1}},
edge/.store in=\contouredge@edge,
begin/.store in=\contouredge@begin,
end/.store in=\contouredge@end,
left bcap/.store in=\contouredge@leftbcap,
right bcap/.store in=\contouredge@rightbcap,
left ecap/.store in=\contouredge@leftecap,
right ecap/.store in=\contouredge@rightecap,
}
\tikzset{
contour edge settings/do stuff/.style 2 args={
	/utils/exec={
	\ifnum\contouredge@isabove=1\tikzset{postaction={
	create contour={direction=#1,radius={\contourradius+\edgehalo},begin=\contouredge@begin,end=\contouredge@end},
	postaction={exec on layer={graph edge \contouredge@layer}{\fill[white] plot[smooth] coordinates \contour -- plot[smooth] coordinates \contourrev -- cycle;}}}}\fi},
	postaction={create contour={direction=#1,radius={\contourradius-.25pt},begin=0,end=\contouredge@begin},
	postaction={exec on layer={graph near vertex mask}{\fill[white] plot[smooth] coordinates \contour -- plot[smooth] coordinates \contourrev -- cycle;}}},
	postaction={create contour={direction=#1,radius={\contourradius-.25pt},begin=\contouredge@end,end=0.998},
	postaction={exec on layer={graph near vertex mask}{\fill[white] plot[smooth] coordinates \contour -- plot[smooth] coordinates \contourrev -- cycle;}}},
	postaction={
	create contour={direction=#1,radius=\contourradius,begin=\contouredge@begin,end=\contouredge@end},
	postaction={contour edge settings/tmpstyle/.style/.expand once=#2,exec on layer={graph edge \contouredge@layer}{\draw[contour edge settings/tmpstyle] plot[smooth] coordinates \contour;},}},
	postaction={
	create contour={direction=#1,radius=\contourradius,begin=0,end=0.998},
	postaction={contour edge settings/tmpstyle/.style/.expand once=#2,exec on layer={graph near vertex contour}{\draw[contour edge settings/tmpstyle] plot[smooth] coordinates \contour;\ifnumcomp{#1}{=}{-1}{\pic[draw,contour edge settings/tmpstyle]{arc from center and point={(contour-begin-onedge)}{(contour-begin)}{-\contouredge@rightbcap}};\pic[draw,contour edge settings/tmpstyle]{arc from center and point={(contour-end-onedge)}{(contour-end)}{\contouredge@rightecap}};}{\pic[draw,contour edge settings/tmpstyle]{arc from center and point={(contour-begin-onedge)}{(contour-begin)}{\contouredge@leftbcap}};\pic[draw,contour edge settings/tmpstyle]{arc from center and point={(contour-end-onedge)}{(contour-end)}{-\contouredge@leftecap}};}
},}},
},
contour edge/.style={
	postaction={
	contour edge settings/.cd,
	defaults,
	#1,
	/tikz/.cd,
	contour edge settings/do stuff={-1}{\contouredge@right},
	contour edge settings/do stuff={1}{\contouredge@left},
	graph edge={\contouredge@layer}{\contouredge@edge},
}}
}

\tikzset{
bunch of nodes settings/.is family,
bunch of nodes settings/.cd,
defaults/.style={first edge style/.style={},last edge style/.style={},first node={},last node={},/utils/exec={\def\bunchofnodes@hasbrace{0}}},
size/.store in=\bunchofnodes@size,
first edge/.code={\pgfkeys{/tikz/bunch of nodes settings/first edge style/.style/.expand once={#1}}},
last edge/.code={\pgfkeys{/tikz/bunch of nodes settings/last edge style/.style/.expand once={#1}}},
first node/.store in=\bunchofnodes@firstnode,
last node/.store in=\bunchofnodes@lastnode,
brace/.code={\def\bunchofnodes@hasbrace{1}\edef\bunchofnodes@bracetext{#1}},
}
\tikzset{
bunch of nodes/.pic={
\tikzset{
bunch of nodes settings/.cd,
defaults,
#1,
/tikz/.cd,
}
\path[bunch of nodes settings/first edge style] (0,0) -- (-30:\bunchofnodes@size) coordinate (-firstnode);
\path[bunch of nodes settings/last edge style] (0,0) -- (30:\bunchofnodes@size) coordinate (-lastnode);
\bgroup
\pic (tmppic) at (-firstnode) {code={\bunchofnodes@firstnode}};
\pic (tmppic)  at (-lastnode) {code={\bunchofnodes@lastnode}};
\egroup
\foreach \i in {0.33,0.5,0.66} {\fill[black] ($(-firstnode)!\i!(-lastnode)$) circle(0.5pt);}
\ifnumcomp{\bunchofnodes@hasbrace}{=}{1}{\draw[auto=left,decorate,decoration={raise={0.1*\bunchofnodes@size+0.4em},brace}] ([yshift=0.2em]-lastnode) -- ([yshift=-0.2em]-firstnode) node[midway,inner sep={0.1*\bunchofnodes@size+1.2em}] {\bunchofnodes@bracetext};}{}
}
}

\def\afterjoin@precision{0.01}
\newlength{\afterjoin@width}
\def\surfyscale{0.4}
\newlength{\tuberadius}
\newlength{\tubesmallradius}
\newlength{\tubelargeradius}
\def\surfdiskradius{2.5}
\def\surfdisksmallradius{1.2}
\def\surfdisksmallsep{0.2}
\newlength{\surroundingradius}
\pgfdeclarelayer{surf surrounding}
\pgfdeclarelayer{surf base}
\pgfdeclarelayer{surf edge behind}
\pgfdeclarelayer{surf tube}
\pgfdeclarelayer{surf edge front}

\tikzset{
surf picture/.style={
/utils/exec={\pgfsetlayers{surf surrounding,surf base,surf edge behind,surf tube,surf edge front,graph vertex,main}},
z={(0,1)},
y={(0,\surfyscale)},
},
surf boundary/.style={draw=black!75,line width=.6pt},
surf boundary soft/.style={surf boundary,opacity=.4},
surf boundary dashed/.style={surf boundary soft,densely dashed},
}

\tikzset{
surf edge/.style 2 args={edge on layer={surf edge #1}{#2}}
}

\tikzset{
pics/surf circle/.style 2 args={code={
\edef\rad{#2}
\expandafter\xdef\csname surfcircle@#1@radius\endcsname{\rad}
\coordinate (#1-center) at (0,0);
}},
pics/surf tube/.style 2 args={code={
\coordinate (tube-lground) at (#1);
\coordinate (tube-rground) at (#2);
\tikzmath{coordinate \v;\v=(tube-lground)-(tube-rground);\rad=veclen(\vx,\vy)/2;}
\global\tubelargeradius=\rad pt
\coordinate (tube-middle) at ($(tube-lground)+(\rad pt,\rad pt)$);
\begin{pgfonlayer}{surf edge front}
\path[surf boundary dashed] \tuberightpath{0}{180} \tubeleftpath{0}{180};
\path[surf boundary] ([xshift=-\tuberadius]tube-rground) arc(0:180:{\rad pt-\tuberadius}) ([xshift=\tuberadius]tube-rground) arc(0:180:{\rad pt+\tuberadius});
\path[surf boundary soft] \tuberightpath{0}{-180} \tubeleftpath{0}{-180};
\end{pgfonlayer}
}},
}
\NewDocumentCommand{\surfcirclepoint}{mm}{($(#1-center)+({#2}:{\csname surfcircle@#1@radius\endcsname})$) }
\NewDocumentCommand{\surfcirclepath}{mmm}{($(#1-center)+({#2}:{\csname surfcircle@#1@radius\endcsname})$) arc({#2}:{#3}:{\csname surfcircle@#1@radius\endcsname}) }
\NewDocumentCommand{\tuberightpoint}{m}{($(tube-rground)+({#1}:{\tuberadius} and {\tubesmallradius})$) }
\NewDocumentCommand{\tubemiddlepoint}{m}{($(tube-middle)+({#1}:{\tubesmallradius} and {\tuberadius})$) }
\NewDocumentCommand{\tubeleftpoint}{m}{($(tube-lground)+({#1}:{\tuberadius} and {\tubesmallradius})$) }
\NewDocumentCommand{\tuberightpath}{mm}{($(tube-rground)+({#1}:{\tuberadius} and {\tubesmallradius})$) arc ({#1}:{#2}:{\tuberadius} and {\tubesmallradius}) }
\NewDocumentCommand{\tubemiddlepath}{mm}{($(tube-middle)+({#1}:{\tubesmallradius} and {\tuberadius})$) arc ({#1}:{#2}:{\tubesmallradius} and {\tuberadius}) }
\NewDocumentCommand{\tubeleftpath}{mm}{($(tube-lground)+({#1}:{\tuberadius} and {\tubesmallradius})$) arc ({#1}:{#2}:{\tuberadius} and {\tubesmallradius}) }
\NewDocumentCommand{\tubecontour}{}{\tuberightpath{0}{-180} arc(0:180:{\tubelargeradius-\tuberadius}) -- \tubeleftpath{0}{-180} arc(180:0:{\tubelargeradius+\tuberadius})}
\NewDocumentCommand{\tuberightcontour}{}{\tuberightpath{0}{-180} arc(0:90:{\tubelargeradius-\tuberadius}) arc ({270}:{90}:{\tubesmallradius} and {\tuberadius}) arc(90:0:{\tubelargeradius+\tuberadius})}
\NewDocumentCommand{\tubeleftcontour}{}{\tubeleftpath{0}{-180} arc(180:90:{\tubelargeradius+\tuberadius}) arc ({90}:{270}:{\tubesmallradius} and {\tuberadius}) arc(90:180:{\tubelargeradius-\tuberadius}) -- cycle}
\NewDocumentCommand{\tubefill}{m}{\begin{pgfonlayer}{surf tube}\fill[#1] \tubecontour{};\end{pgfonlayer}}
\NewDocumentCommand{\tuberightfill}{m}{\begin{pgfonlayer}{surf tube}\fill[#1] \tuberightcontour{};\end{pgfonlayer}}
\NewDocumentCommand{\tubeleftfill}{m}{\begin{pgfonlayer}{surf tube}\fill[#1] \tubeleftcontour{};\end{pgfonlayer}}
\NewDocumentCommand{\tubebelt}{mm}{\begin{pgfonlayer}{surf edge front}\path[#2] \tubemiddlepath{-90}{90};\path[#1] \tubemiddlepath{90}{270};\end{pgfonlayer}}

\tikzset{
surrounding/.style={edge on layer={surf surrounding}{draw,#1,line width=\surroundingradius,line cap=round}}
}

\tikzset{
after join/.style n args={3}{
/utils/exec={
\xdef\afterjoin@path{}
\xdef\afterjoin@midpath{}
\xdef\afterjoin@pathrev{}
\xdef\afterjoin@midpathrev{}
},
decorate,decoration={markings,mark=between positions 0 and .998 step \afterjoin@precision with {
\path let \n1={int(\pgfkeysvalueof{/pgf/decoration/mark info/sequence number}-1)},\n2={\n1*\afterjoin@precision},\n3={sqrt(1-\n2)},\n4={\n2>0.2} in (0,{\n3*\afterjoin@width}) coordinate (afterjoin-\n1) (0,{-\n3*\afterjoin@width}) coordinate (afterjoin-rev-\n1) \pgfextra{
	\xdef\afterjoin@path{\afterjoin@path (afterjoin-\n1)}
	\xdef\afterjoin@pathrev{(afterjoin-rev-\n1) \afterjoin@pathrev}
	\ifnum\n4=1
		\xdef\afterjoin@midpath{\afterjoin@midpath (afterjoin-\n1)}
		\xdef\afterjoin@midpathrev{(afterjoin-rev-\n1) \afterjoin@midpathrev}
	\fi
};
},mark=at position 0.999 with {
\coordinate (afterjoin-end);
\xdef\afterjoin@path{\afterjoin@path (afterjoin-end)}
\xdef\afterjoin@pathrev{(afterjoin-end) \afterjoin@pathrev}
\xdef\afterjoin@midpath{\afterjoin@midpath (afterjoin-end)}
\xdef\afterjoin@midpathrev{(afterjoin-end) \afterjoin@midpathrev}
}},
postaction={
exec on layer={surf edge #1}{
\fill[#3,save path=\afterjoinpath] plot[smooth] coordinates \afterjoin@path -- plot[smooth] coordinates \afterjoin@pathrev let \p1=(afterjoin-0),\p2=(afterjoin-rev-0),\p3=($(\p2)-(\p1)$),\n1={veclen(\x3,\y3)/2},\n2={atan2(\y3,\x3)} in arc(\n2:\n2-180:\n1);
\begin{scope}
\clip \surfcirclepath{#2}{0}{360};
\path[surf boundary,use path=\afterjoinpath];
\end{scope}
\begin{scope}
\begin{pgfinterruptboundingbox}
\clip[insert path={(-10cm,-10cm) rectangle (10cm,10cm)}] \surfcirclepath{#2}{0}{360};
\path[surf boundary] plot[smooth] coordinates \afterjoin@midpath -- plot[smooth] coordinates \afterjoin@midpathrev;
\end{pgfinterruptboundingbox}
\end{scope}
}
}
}
}

\tikzset{
cmove setting one disk/.pic={\path[use as bounding box] (-3.25,-4.5) rectangle(3.25,5);\pic {surf circle={d#1}{\surfdiskradius}};\tikzset{exec on layer={surf base}{\fill[disk #1] \surfcirclepath{d#1}{0}{360};},exec on layer={surf edge behind}{\path[surf boundary] \surfcirclepath{d#1}{0}{360};}}},
cmove setting one disk tube/.pic={\path \surfcirclepoint{d#1}{180} coordinate (-l) \surfcirclepoint{d#1}{0} coordinate (-r);\pic {surf tube={$(-l)!.25!(-r)$}{$(-l)!.75!(-r)$}};},
cmove setting two disks/.pic={\path[use as bounding box] (-2.6,-2.3) rectangle (2.6,5);\foreach \i/\j in {1/-1,2/1} {
\pic at ({\j*(\surfdisksmallradius+\surfdisksmallsep)},0) {surf circle={d\i}{\surfdisksmallradius}};
\tikzset{exec on layer={surf base}{\fill[disk \i] \surfcirclepath{d\i}{0}{360};},exec on layer={surf edge behind}{\path[surf boundary] \surfcirclepath{d\i}{0}{360};}}
}},
cmove setting two disks tube/.pic={\path \surfcirclepoint{d1}{180} coordinate (-l) \surfcirclepoint{d2}{0} coordinate (-r);\pic {surf tube={{-\surfdiskradius/2},0}{{+\surfdiskradius/2},0}};}
}

\def\torus@precision{81}
\tdplotsetmaincoords{65}{0}
\def\radiusA{1.7}
\def\radiusB{0.5}

\def\torusprecision{\torus@precision}
\tikzset{declare function={%
torusx(\u,\v)=cos(\u)*(\radiusA + \radiusB*cos(\v)); 
torusy(\u,\v)=(\radiusA + \radiusB*cos(\v))*sin(\u);
torusz(\u,\v)=\radiusB*sin(\v);
vcrit1(\u)=atan(tan(\tdplotmaintheta)*sin(\u));
vcrit2(\u)=180+atan(tan(\tdplotmaintheta)*sin(\u));
thetacritA(\asdsd)=atan(sqrt(\radiusA/\radiusB-1));
thetacritB(\asdasd)=acos(\radiusB/\radiusA);
ucritA(\th)=180+(90/pi)*sqrt(abs(-(\radiusA^2*pow(cot(\th),2))+4*pow(\radiusB,2)/pow(sin(2*\th),2)))/\radiusA; 
ucritB(\th)=540-ucritA(\th);
umaxA(\th)=asin(sqrt(abs(-pow(cot(\th),2)+4*pow(\radiusB,2)/(pow((sin(2*\th)*\radiusA),2)))));
umaxB(\th)=180-umaxA(\th);}}
\tikzset{3d torus/.style={/utils/exec=\pgfmathsetmacro{\DDA}{int(sign(sin(thetacritA(1))-sin(\tdplotmaintheta)))}
  \pgfmathsetmacro{\DDB}{int(sign(sin(thetacritB(1))-sin(\tdplotmaintheta)))},
  insert path={
  plot[variable=\x,domain=1:359,smooth cycle,samples=\torus@precision]
  ({torusx(\x,vcrit1(\x))},
 {torusy(\x,vcrit1(\x))},
 {torusz(\x,vcrit1(\x))}) 
   \ifnum\DDA=1
    plot[variable=\x,domain=0:360,smooth cycle,samples=\torus@precision]
    ({torusx(\x,vcrit2(\x))},
    {torusy(\x,vcrit2(\x))},
    {torusz(\x,vcrit2(\x))})    
   \else
   \ifnum\DDB=1 
    plot[variable=\x,domain={umaxA(\tdplotmaintheta)}:{umaxB(\tdplotmaintheta)},smooth,samples=\torus@precision]
    ({torusx(\x,vcrit2(\x))},
    {torusy(\x,vcrit2(\x))},
    {torusz(\x,vcrit2(\x))})    --
    plot[variable=\x,domain={180+umaxA(\tdplotmaintheta)}:{180+umaxB(\tdplotmaintheta)},smooth,samples=\torus@precision]
    ({torusx(\x,vcrit2(\x))},
    {torusy(\x,vcrit2(\x))},
    {torusz(\x,vcrit2(\x))})  -- cycle  
    \fi 
   \fi
  }},3d torus stretch/.style={/utils/exec=\pgfmathsetmacro{\DDA}{int(sign(thetacritA()-\tdplotmaintheta))},
  insert path={\ifnum\DDA=-1
   plot[variable=\x,domain={ucritA(\tdplotmaintheta)}:{ucritB(\tdplotmaintheta)},smooth,samples=\torus@precision]
    ({torusx(\x,vcrit2(\x))},
    {torusy(\x,vcrit2(\x))},
    {torusz(\x,vcrit2(\x))}) 
  \fi 
}}}
\newcommand{\ontorus}[2]{{torusx(#1,#2)},{torusy(#1,#2)},{torusz(#1,#2)}}
\tikzset{torus straight/.style n args={4}{insert path={plot[variable=\t,domain=0:1] (\ontorus{interp(#1,#3,\t)}{interp(#2,#4,\t)})}},
atorus straight/.style n args={4}{insert path={-- plot[variable=\t,domain=0:1] (\ontorus{interp(#1,#3,\t)}{interp(#2,#4,\t)})}},
}

\def\drawtorusgraph{
\draw[black edge] plot[variable=\x,domain={vcrit2(240)}:{vcrit1(240)}] (\ontorus{240}{\x});
\draw[black edge] [torus straight={240}{60}{270}{30}];
\draw[black edge] [torus graph small region];
\draw[black edge] [torus straight={-50}{30}{80}{150}];
\draw[black edge] [torus straight={80}{150}{240}{60}];
\pic at (\ontorus{240}{60}) {black vertex};
\pic at (\ontorus{240}{0}) {white vertex};
\pic at (\ontorus{270}{30}) {white vertex};
\pic at (\ontorus{310}{30}) {black vertex};
\pic at (\ontorus{80}{150}) {white vertex};
}
\tikzset{
torus graph small region/.style={insert path={
plot[variable=\t,domain=-1:1] (\ontorus{interp(270,310,.5*\t+.5)}{30+45*sqrt(1-\t*\t)}) -- plot[variable=\t,domain=1:-1] (\ontorus{interp(270,310,.5*\t+.5)}{30-45*sqrt(1-\t*\t)})
}},
torus graph upper contour/.style={insert path={
	plot[variable=\x,domain={vcrit2(240+\xac):{60+\yac}}] (\ontorus{240+\xac}{\x})
	[atorus straight={240+\xac}{60+\yac}{270-\xac}{30+\yac}]
	-- plot[variable=\t,domain=-1:1] (\ontorus{interp(270-\xac,310+\xac,.5*\t+.5)}{30+\yac+45*sqrt(1-\t*\t)})
	[atorus straight={-50+\xac}{30+\yac}{80}{150+\yac}]
	[atorus straight={80}{150+\yac}{240-\xac}{60+\yac}]
	-- plot[variable=\x,domain={60+\yac}:{vcrit2(240-\xac)}] (\ontorus{240-\xac}{\x})
}},
torus graph lower contour/.style={insert path={
	plot[variable=\x,domain={vcrit1(240-\xac)}:{60-\yac}] (\ontorus{240-\xac}{\x})
	[atorus straight={240-\xac}{60-\yac}{80}{150-\yac}]
	[atorus straight={80}{150-\yac}{-50+\xac}{30-\yac}]
	-- plot[variable=\t,domain=1:-1] (\ontorus{interp(270-\xac,310+\xac,.5*\t+.5)}{30-\yac-45*sqrt(1-\t*\t)})
	[atorus straight={270-\xac}{30-\yac}{240+\xac}{60-\yac}]
	-- plot[variable=\x,domain={60-\yac}:{vcrit1(240+\xac)}] (\ontorus{240+\xac}{\x})
}},
torus graph behind/.style={insert path={
plot[variable=\x,domain={360+vcrit1(240)}:{vcrit2(240)}] (\ontorus{240}{\x})
}},
torus graph behind right/.style={insert path={
plot[variable=\x,domain={360+vcrit1(240+\xac)}:{vcrit2(240+\xac)}] (\ontorus{240+\xac}{\x})
}},
torus graph behind left/.style={insert path={
plot[variable=\x,domain={vcrit2(240-\xac)}:{360+vcrit1(240-\xac)}] (\ontorus{240-\xac}{\x})
}},
torus graph small contour/.style={insert path={
plot[variable=\t,domain=-1:1] (\ontorus{interp(270+\xac,310-\xac,.5*\t+.5)}{30+(45-\yac)*sqrt(1-\t*\t)}) -- plot[variable=\t,domain=1:-1] (\ontorus{interp(270+\xac,310-\xac,.5*\t+.5)}{30-(45-\yac)*sqrt(1-\t*\t)})
}}
}

\ifnum\graphicfast=1
\tikzset{contour edge/.style={
	postaction={
	contour edge settings/.cd,
	defaults,
	#1,
	/tikz/.cd,
	graph edge={\contouredge@layer}{\contouredge@edge},
}}}
\fi

\makeatother
\usetikzlibrary{external}
\tikzset{external/only named=true}

\usepackage{booktabs}
\usepackage{tabularx}
\usepackage{ltablex}

\begin{document}

\title{Solution of the Hurwitz problem\\ with a length-2 partition}

\author{Filippo~\textsc{Baroni}\and
Carlo~\textsc{Petronio}\thanks{Partially supported by INdAM through GNSAGA, by
MUR through the PRIN project
n.~2017JZ2SW5$\_$005 ``Real and Complex Manifolds: Topology, Geometry and Holomorphic Dynamics''
and by UniPI through the PRA$\_2018\_22$ ``Geometria e Topologia delle Variet\`a''}}

\maketitle

\begin{abstract}\noindent
In this note we provide a new partial solution to the Hurwitz existence problem
for surface branched covers. Namely, we consider candidate branch data with base surface the sphere
and one partition of the degree having length two, and we fully determine which of them are realizable and which are exceptional.
The case where the covering surface is also the sphere was solved
somewhat recently by Pakovich, and we deal here with the case
of positive genus. We show that the only other exceptional candidate data, besides  those of Pakovich (five infinite families and one sporadic case),
are a well-known very specific infinite family in degree $4$
(indexed by the genus of the candidate covering surface, which can attain any value), five sporadic cases (four in genus $1$ and one in genus $2$),
and another infinite family in genus $1$ also already known.
Since the degree is a composite number for all these exceptional data, our findings provide more evidence for the prime-degree conjecture.
Our argument proceeds by induction on the genus and on the number of branching points, so our results logically depend on those of Pakovich,
and we do not employ the technology of constellations on which his proof is based.

\smallskip

\end{abstract}

\noindent
A \emph{surface branched cover} is a map $f:\Theta\to\Sigma$, where $\Theta$ and $\Sigma$ are connected closed
and orientable surfaces and $f$ is locally modeled on
the function $\matC\ni z\mapsto z^k\in\matC$ for a positive integer $k$. If $k>1$ the point corresponding to $0$ in the target $\matC$ is called a \emph{branching point}, and $k$
is called the \emph{local degree} of $f$ at the point corresponding to $0$ in the source $\matC$. There is a finite number $n$
of branching points, and removing all of them and their preimages one gets an ordinary cover of some degree $d$, called the \emph{degree} of $f$.
The collection of the local degrees at the preimages of the $j$-th branching point forms a partition $\pi_j$ of $d$, namely an unordered array of positive integers summing up to $d$, possibly with repetitions. To such an $f$ we associate a symbol called a \emph{branch datum} given by
$\left(\Theta,\Sigma,d,n;\pi_1,\ldots,\pi_n\right)$
where the partitions $\pi_1,\ldots,\pi_n$ are viewed up to reordering. If $\ell_j$ denotes the length of $\pi_j$
(\emph{i.e.}, the number of entries of $\pi_j$), the datum satisfies the Riemann-Hurwitz condition
\begin{equation}
\label{general:RH}
\chi\left(\Theta\right)-\left(\ell_1+\ldots+\ell_n\right)=d\cdot\left(\chi\left(\Sigma\right)-n\right).
\end{equation}
We now call \emph{candidate branch datum} a symbol
$\left(\Theta,\Sigma,d,n;\pi_1,\ldots,\pi_n\right)$
with $\Theta$ and $\Sigma$ connected closed orientable surfaces, $d$ and $n$ positive integers, and $\pi_1,\ldots,\pi_n$ partitions of $d$, satisfying condition~(\ref{general:RH}).
We will always assume that no $\pi_j$ is the trivial partition with all entries $1$.
A candidate branch datum is \emph{realizable} if it is the branch datum associated to an existing surface branched cover $f$, and \emph{exceptional} otherwise.

The question of characterizing which candidate branch data are realizable and which are exceptional is known as the \emph{Hurwitz existence problem}~\cite{Hur1891}.
It has a long history (see the surveys~\cite{Bologna, Hurwsurvey}
and~\cite{Bar01, BeEd84, CoZa18, Ez78, FePe18, Ger87, Hus62, KhoZdr96, Pako, PaPe, PaPebis, PePe06, PePe08, SX10, Thom65}),
and many motivations (see for instance~\cite{MSS}).
Before proceeding we fix the following:

\paragraph{Notation}
We indicate by $S$ the sphere, by $T$ the
torus and by $g\cdot T$ the orientable connected closed surface of genus $g$ for $g\geqslant2$, but we also accept the symbols $0\cdot T=S$ and $1\cdot T=T$.
We use double square brackets $\partition{\,\cdot\,}$ to denote an unordered array of objects with possible repetitions.
So a partition $\pi$ of a positive integer $d$ is given by $\pi=\partition{q_1,\ldots,q_m}$ where the $q_i$'s are positive integers and $q_1+\ldots+q_m=d$.
We denote by $\ell(\pi)$ the length $m$ of $\pi$.

\paragraph{Some results from the literature}
We cite a crucial known~\cite{EKS} result:

\begin{thm}
Every candidate branch datum
$\left(\Theta,g\cdot T,d,n;\pi_1,\ldots,\pi_n\right)$ with $g\geqslant 1$
is realizable.
\end{thm}

This implies that to find a full solution of the Hurwitz existence problem one is only left to consider the case
where the candidate covered surface is the sphere $S$ (see also Remark~\ref{non-ori:rem} below). Many different techniques were
employed over time to attack the problem in this case, and a huge variety of exceptional and realizable candidate branch data were detected
(see the reference already cited above). We refrain from giving a full account of these results here, confining ourselves to those
that are most relevant for the present paper.
We start with the following (see~\cite{Thom65} for the case $g=0$ and~\cite[Proposition 5.2]{EKS} for any $g$):

\begin{thm}\label{length-1:thm}
A candidate branch datum $(g\cdot T,S,d,n;\pi_1,\ldots,\pi_n)$ such that $\ell(\pi_j)=1$ for some $j$, that is $\pi_j=\partition{d}$, is always realizable.
\end{thm}

One can informally phrase this result saying that a candidate branch datum with one partition
\emph{as short as it could at most be} is realizable. It is then
natural to consider the case where one partition has the next shortest possible length, \textit{i.e.}~$2$.
Namely one can ask the question of which candidate branch data $(g\cdot T,S,d,n;\pi_1,\ldots,\pi_n)$
such that $\ell(\pi_j)=2$ for some $j$, that is $\pi_j=\partition{s,d-s}$ for $0<s<d$, are realizable. The solution was obtained in the following cases:
\begin{itemize}
\item For any $g$, any $n$ and $s=1$ in~\cite{EKS};
\item For $n=3$, any $s$ and $\pi_2=\pi_3=\partition{2,\ldots,2}$, whence $g=0$, in~\cite{EKS};
\item For any $g$, $n=3$ and $s=2$ in~\cite{PePe08};
\item For $g=0$, any $n$ and any $s$ in~\cite{Pako}.
\end{itemize}
The aim of the present paper is to provide a complete answer to the question, namely to face the case of any $g\geqslant1$,
any $n$ and any $s$ left out by Pakovich~\cite{Pako}. The (long) answer will be given in Section~\ref{statement:sec} (see Theorem~\ref{main:thm},
whose statement does not require any of the notions treated in Sections~\ref{dessins:sec} and~\ref{monodromy:sec}).
We only mention here that our argument is based on two inductions, one on $g$ for fixed $n=3$, with the base step $g=0$ given by the statement of~\cite{Pako},
and one on $n$. In particular, our proof depends on that of Pakovich for $n=3$, and it does not
employ the technology of constellations he uses.

\paragraph{The prime-degree conjecture}
The following was proposed in~\cite{EKS} and served ever since as a guiding idea in this area of research:

\begin{conj}
A candidate branch datum $(g\cdot T,S,d,n;\pi_1,\ldots,\pi_n)$ with prime $d$ is always realizable.
\end{conj}

It was also shown in~\cite{EKS} that proving the conjecture with $n=3$ would imply its validity for all $n$.
All the results obtained so far concerning the Hurwitz existence problem turned out to be compatible with this
conjecture (and some of them actually provided striking supporting evidence for it, see for instance~\cite{PaPe}).
Our Theorem~\ref{main:thm} below makes no exception.

\paragraph{Concluding remarks}
We end this introduction with some additional considerations.

Among the various approaches to the Hurwitz existence problem, a remarkable one was proposed (for arbitrary $n$) by Zheng~\cite{Zh06} in computational terms,
that he implemented on a machine for $n=3$, giving a complete classification of realizable and exceptional candidate branch data with $n=3$ and $d\leqslant20$.
The first named author pushed 
forward the implementation of the methods of~\cite{Zh06}, 
obtaining several computer-aided findings, some of which
are also referred to in the present paper 
in Sections~\ref{threepoints:sec}
and~\ref{morepoints:sec} 
(see the Appendix for details). However, our results do not 
really depend on these findings: we only use them
to make our arguments more concise, but the realizability 
and exceptionality of the relevant candidate branch data 
could always also be
very easily established by theoretical methods.

\begin{rem}\label{non-ori:rem}
\emph{The Hurwitz existence problem can be stated also for candidate branch data
$\calD=\left(\Theta,\Sigma,d,n;\pi_1,\ldots,\pi_n\right)$
with non-orientable $\Sigma$ and possibly non-orientable $\Theta$.
In this case the Riemman-Hurwitz condition~(\ref{general:RH}) must be complemented with two more (one obvious for orientable $\Theta$ and one less
obvious for non-orientable $\Theta$).  However the following facts proved in~\cite{EKS}, with $\matP$ denoting
the projective plane, show that a full solution of the problem follows
once it is given for $\Sigma=S$:
\begin{itemize}
\item If $\Sigma=k\cdot\matP$ with $k\geqslant 2$ then $\calD$ is realizable;
\item If $\Sigma=\matP$ and $\Theta$ is non-orientable then $\calD$ is realizable;
\item If $\Sigma=\matP$ and $\Theta$ is orientable, $\calD$ is realizable if and only if for $j=1,\ldots,n$ the partition $\pi_j$ of $d$ can be written as
the juxtaposition of two partitions $\pi'_j,\pi''_j$ of $d/2$ so that
$\left(\Theta,S,d/2,2n;\pi'_1,\pi''_1,\ldots,\pi'_n,\pi''_n\right)$
is realizable.
\end{itemize}}
\end{rem}

\begin{rem}\emph{The original question of Hurwitz was actually that of \emph{how many} realizations of a given candidate datum exist, up
to some natural equivalence relation. This problem was given a deep but somewhat implicit solution in~\cite{Medn81, Medn84}, from which
it is not easy to extract a solution to the realizability problem for specific candidate branch data. See also~\cite{KM04, KM07, KML05},
that also provide general but indirect answers, the more explicit~\cite{x1, x3, odddegree} and the easy remarks contained in~\cite{PS19} on
the different ways the realizations can be counted.}
\end{rem}

The paper is organized as follows. In Section~\ref{dessins:sec} we introduce the idea of \emph{reduction move}, which is crucial
for our induction arguments, and we define some \emph{topological} reduction moves, based on the notion of \emph{dessin d'enfant}. In Section~\ref{monodromy:sec}
we next introduce some \emph{algebraic} reduction moves, using the monodromy approach to the Hurwitz existence problem.
Then in Section~\ref{statement:sec} we state our result, listing all the exceptional candidate branch data with a length-$2$ partition,
and explaining why they are indeed exceptional using the technology of Sections~\ref{dessins:sec} and~\ref{monodromy:sec}.
As a conclusion, in Sections~\ref{threepoints:sec} and~\ref{morepoints:sec} respectively, we carry out the induction arguments on $g$ (for $n=3$) and on $n$,
thereby showing that the candidate branch data $(g\cdot T,S,d,n;\pi_1,\ldots,\pi_n)$ with some $\ell(\pi_j)=2$ and not listed in Theorem~\ref{main:thm} are indeed realizable.


\section{Dessins d'enfant\\ and topological reduction moves}\label{dessins:sec}

As already mentioned, our proof of Theorem~\ref{main:thm}
for candidate branch data
$(g\cdot T,S,d,n;\pi_1,\ldots,\pi_{n-1},\partition{s,d-s})$
relies on two induction arguments, one on $g$ for $n=3$, and then one on $n$.
The core ingredient of both arguments is the following notion:

\begin{defn}
\emph{Let
$\calD=(g\cdot T,S,d,n;\pi_1,\ldots,\pi_n)$
be a symbol involving generic integers $g,d,n$ and partitions $\pi_1,\ldots,\pi_n$ of $d$ subject to a set $\calC$ of conditions. Let
$\calD'=(g'\cdot T,S,d',n';\pi'_1,\ldots,\pi'_{n'})$
be a similar symbol, where $g',d',n',\pi'_1,\ldots,\pi'_{n'}$ depend on $g,d,n,\pi_1,\ldots,\pi_n$.
We say that there is a \emph{reduction move} $\calD\leadsto\calD'$ subject to $\calC$ if
the following happens:
\begin{itemize}
\item Whenever $\calD$ is a candidate branch datum and $g,d,n,\pi_1,\ldots,\pi_n$ satisfy $\calC$ then
$\calD'$ is also a candidate branch datum;
\item The realizability of $\calD'$ implies that of $\calD$.
\end{itemize}}
\end{defn}

To describe our reduction moves we add an extra bit of notation. For $\eta,\rho$ partitions of integers $p,q$ we
define $\eta*\rho$ as the partition of $p+q$ obtained by juxtaposing $\eta$ and $\rho$ (recall that the order is immaterial).
Moreover when writing a reduction move $\calD\leadsto\calD'$
we will highlight using a boldface character those elements
of $g,d,n,\pi_1,\ldots,\pi_n$ in $\calD$ and those of $g',d',n',\pi'_1,\ldots,\pi'_{n'}$ in $\calD'$
that do not coincide with the corresponding elements of the other symbol.

\paragraph{Dessins d'enfant}
We review here a notion popularized by Grothendieck in~\cite{Groth} (see also~\cite{Cohen} and the
more general~\cite{LZ}), recalling its connection with the Hurwitz existence problem.

\begin{defn}\emph{We call \emph{dessin d'enfant} a finite graph $\Gamma$ in a surface $\Sigma$ such that:
\begin{itemize}
\item $\Gamma$ is bipartite, namely each of its vertices is colored black or white;
\item Each edge of $\Gamma$ has a black and a white end;
\item $\Sigma\setminus\Gamma$ is a union of topological open discs, called \emph{regions}.
\end{itemize}
If $R$ is a region, the \emph{length} of $R$ is half the number of edges of $\Gamma$ adjacent to $R$,
counted twice if $R$ is incident from both sides.
Note that the length of a complementary region can also be defined as the number of black (or, equivalently, white) vertices adjacent to it,
counted with multiplicity.}
\end{defn}

\begin{prop}\label{dessins:prop}
A candidate branch datum $\calD=(\Sigma,S,d,3;\pi_1,\pi_2,\pi_3)$ is realizable if and only if there exists a dessin d'enfant in $\Sigma$ with
valences of the black (respectively, white) vertices given by the entries of $\pi_1$ (respectively, $\pi_2$), and lengths
of the complementary regions given by the entries of $\pi_3$.
\end{prop}

The ``only if'' part of the statement is obtained by defining $\Gamma$ as $f^{-1}(e)$, where $f$ is a branched cover realizing $\calD$
and $e$ is a segment with ends $p_1$ and $p_2$ and avoiding $p_3$, where $p_j\in S$ is the branching point corresponding to $\pi_j$,
with $f^{-1}(p_1)$ black and $f^{-1}(p_2)$ white. The ``if'' part is achieved by reversing this construction, see~\cite{LZ}.

\paragraph{Collapse of segments}
We now describe a key ingredient underlying the reduction moves described in the rest of this section.
Let $\Gamma$ be a bipartite graph in a surface $\Sigma$ and let $c$ be a
segment in $\Sigma$ such that $\Gamma\cap c$ consists of two distinct vertices of $\Gamma$ with the same colour.
Then we can construct a new bipartite graph $\Gamma'$
in $\Sigma$ as that obtained from $\Gamma$ by collapsing $c$ to a point, thus fusing together
its ends, as in Fig.~\ref{collapse:def:fig}. Note that the complementary regions of $\Gamma'$ naturally correspond to those of $\Gamma\cup c$.
\begin{figure}[htbp]
\faifig{Quella a pag 33 della tesi in alto}
{dessin-join-example}
{Collapse of a curly segment\label{collapse:def:fig}}
\end{figure}
This construction extends to the case where $c$ is
a finite family of segments $c_1,\ldots,c_k$ with pairwise disjoint interiors, by performing the collapses successively,
provided the ends of each $c_j$ stay distinct after $c_1,\ldots,c_{j-1}$ have been collapsed. Note that in our figures
the segments to be collapsed are drawn as curly arcs, and that we will actually view figures with curly arcs as
if these segments were already collapsed (whence the equality sign in Fig.~\ref{collapse:def:fig}), so we will not employ any
specific notation to express the fact that $\Gamma'$ is a function of $\Gamma$ and $c$.

\paragraph{Convention on figures}
A complementary region $R$ of a dessin d'enfant in a surface $\Sigma$ can be incident to itself along the boundary
(in particular, its closure in $\Sigma$ can fail to be a closed disc). However in our figures we will always unwind $R$,
representing it as a closed disc with portions of $\Gamma$ on its boundary only, as in the example of Fig.~\ref{unwind:fig}.
\begin{figure}[htbp]
\faifig{Quella a pag della tesi 32; bisogna scambiare fra loro le due parti e togliere il segno in mezzo}
{dessin-unwinding-example}
{Left: a dessin d'enfant on the torus with a light grey and a dark grey complementary regions.
Right: how to unwind the light grey region\label{unwind:fig}}
\end{figure}

\paragraph{Trivial partitions}
Recall our convention that in a candidate branch datum $\calD=(g\cdot T,S,d,n;\pi_1,\ldots,\pi_n)$ all $\pi_j$'s should be non-trivial.
However, if a symbol $\calD$ satisfies the Riemann-Hurwitz condition~(\ref{general:RH}) and contains some $\pi_j=\partition{1,\ldots,1}$,
a candidate branch datum is obtained from $\calD$ be removing these $\pi_j$'s and correspondingly reducing $n$. We now have the following easy consequence of~(\ref{general:RH}):

\begin{rem}\label{nontrivial:rem}
\emph{No candidate branch datum
$(g\cdot T,S,d,n;\pi_1,\ldots,\pi_n)$ exists with $\ell(\pi_n)=2$ and $n\leqslant 2$.}
\end{rem}

\paragraph{Topological reduction moves}
We now state and prove the existence of reduction moves $T_1,T_2,T_3,T_4$ that we call
\emph{topological}, both because they allow to lower the genus of a candidate branch datum with $n=3$,
and because they are established using dessins d'enfant.
We introduce these moves $T_j:\calD\leadsto\calD'$ in Propositions~\ref{T1:move:prop} to~\ref{T4:move:prop}.
In all four cases it is immediate to check that if $\calD$ is a candidate branch datum, so it satisfies~(\ref{general:RH}),
then $\calD'$ also satisfies~(\ref{general:RH}). This implies that $\calD'$ is also a 
candidate branch datum thanks to Remark~\ref{nontrivial:rem}, because if one of the partitions in $\calD'$ were trivial 
then $\calD'$ would give rise to a candidate branch datum with less than three partitions and one of length two.

\begin{prop}[Move $T_1$]\label{T1:move:prop}
For $g\geqslant1$ the following is a reduction move:
$$\begin{array}{cl}
T_1: & \calD=(\matebold{g}\cdot T,S,d,3;\partition{1,1,\mathbf{3}}*\rho_1,\pi_2,\partition{s,d-s})\\
\leadsto & \calD'=((\bm{g-1})\cdot T,S,d,3;\partition{1,1,\mathbf{1,1,1}}*\rho_1,\pi_2,\partition{s,d-s}).
\end{array}$$
\end{prop}

\begin{proof}
Assume that $\calD'$ is realizable and let $\Gamma'$ realize it according to Proposition~\ref{dessins:prop}.
Since $\pi_1'$ contains five $1$'s, at least three $1$'s are incident to at least one complementary region of $\Gamma'$, say the light grey one $R$,
as in part 0 of Fig.~\ref{T1move:fig}.
\begin{figure}[htpb]
\faifig{Quella a cavallo delle pagine 34-35 della tesi}
{combinatorial-move-b-1}
{How to realize the reduction move $T_1$ at the level of dessins d'enfant. In part $2$ we show for the last time both the curly arcs that must be viewed as already
collapsed and the result of this collapse. We will not do this again in the next figures\label{T1move:fig}}
\end{figure}
We then proceed as follows:
\begin{enumerate}
\item We attach to $(g-1)\cdot T$ a $1$-handle with both attaching discs inside $R$ (see part 1 of Fig.~\ref{T1move:fig});
\item We collapse along two curly segments as in part 2 of Fig.~\ref{T1move:fig}.
\end{enumerate}
The resulting bipartite graph $\Gamma$ is a dessin d'enfant realizing $\calD$.
\end{proof}

\begin{prop}[Move $T_2$]\label{T2:move:prop}
For $g\geqslant1$, $x\geqslant4$, $2\leqslant s\leqslant d-2$, $x_1,x_2\geqslant1$ and $x_1+x_2=x-2$ the following is a reduction move:
$$\begin{array}{cl}
T_2: & \calD=(\matebold{g}\cdot T,S,\matebold{d},3;\partition{\matebold{x}}*\rho_1,\partition{\mathbf{2}}*\rho_2,\partition{\matebold{s},\bm{d-s}})\\
\leadsto & \calD'=((\bm{g-1})\cdot T,S,\bm{d-2},3;
    \partition{\matebold{x}_\mathbf{1},\matebold{x}_\mathbf{2}}*\rho_1,\rho_2,\partition{\bm{s-1},\bm{d-s-1}}).
\end{array}$$
\end{prop}

\begin{proof}
Take a dessin d'enfant $\Gamma'$ realizing $\calD'$. We consider two cases, showing in both of them how to construct from $\Gamma'$ a dessin d'enfant realizing $\calD$.

\medskip

\noindent\textsc{Case 1:}
\emph{The black vertex of valence $x_1$ is adjacent to one complementary region of $\Gamma'$ and that of valence $x_2$ is adjacent to the other one}.
We are in the situation of part 0 in Fig.~\ref{T2move:case1:fig}, and we proceed as follows, always referring to Fig.~\ref{T2move:case1:fig}:
\begin{figure}[htbp]
\faifig{Quella alle pagine 35 in basso e 36 in alto}
{combinatorial-move-b-2-1}
{The move $T_2$ at the level of dessins d'enfant. Case 1\label{T2move:case1:fig}}
\end{figure}
\begin{enumerate}
\item We attach to $(g-1)\cdot T$ a $1$-handle with attaching discs inside the two complementary regions, as in part 1;
\item We add to $\Gamma'$ the co-core of the $1$-handle, in the form of a black and a white vertex joined by two arcs, as in part 2;
\item We collapse along two curly segments, as in part 3.
\end{enumerate}

\medskip

\noindent\textsc{Case 2:} \emph{The two black vertices of valences $x_1$ and $x_2$ are completely surrounded by one of the complementary regions of $\Gamma'$}
(the light grey one $R$, say). Note that there must be an edge $e$ of $R$ separating it from the other region, as in part 0 of Fig.~\ref{T2move:case2:fig}. Then:
\begin{figure}[htbp]
\faifig{Quella alla pagina 36 in basso}
{combinatorial-move-b-2-2}
{The move $T_2$ at the level of dessins d'enfant. Case 2\label{T2move:case2:fig}}
\end{figure}
\begin{enumerate}
\item We add one black vertex and one white vertex on $e$ (part 1);
\item We attach to $(g-1)\cdot T$ a $1$-handle with attaching discs in $R$ (part 2);
\item We collapse along two curly segments as in part 3.\qedhere
\end{enumerate}
\end{proof}

\begin{prop}[Move $T_3$]\label{T3:move:prop}
For $g\geqslant1$, $x,y\geqslant3$ and $3\leqslant s\leqslant d-3$ the following is a reduction move:
$$\begin{array}{cl}
T_3: & \calD=(\matebold{g}\cdot T,S,\matebold{d},3;\partition{\matebold{x},\matebold{y}}*\rho_1,\partition{\bm{2},\bm{2}}*\rho_2,\partition{\matebold{s},\matebold{d-s}})\\
\leadsto & \calD'=((\bm{g-1})\cdot T,S,\bm{d-4},3;
    \partition{\bm{x-2},\bm{y-2}}*\rho_1,\rho_2,\partition{\bm{s-2},\bm{d-s-2}}).
\end{array}$$
\end{prop}

\begin{proof}
The structure of the proof is similar to the previous one.

\medskip

\noindent\textsc{Case 1:}
\emph{The black vertex of valence $x-2$ is adjacent to one complementary region of $\Gamma'$ and that of valence $y-2$ is adjacent to the other one},
as in part 0 of Fig.~\ref{T3move:case1:fig}. Then:
\begin{figure}[htbp]
\faifig{Quella a pagina 37}
{combinatorial-move-b-3-1}
{The move $T_3$ at the level of dessins d'enfant. Case 1\label{T3move:case1:fig}}
\end{figure}
\begin{enumerate}
\item We attach to $(g-1)\cdot T$ a $1$-handle with attaching discs inside the two complementary regions, as in part 1;
\item We add to $\Gamma'$ the co-core of the $1$-handle, in the form of two black and two white vertices joined by four arcs, as in part 2;
\item We collapse along two curly segments as in part 3.
\end{enumerate}

\medskip

\noindent\textsc{Case 2:} \emph{The two black vertices of valences $x-2$ and $y-2$ are completely surrounded by one of the complementary regions of $\Gamma'$}
(the light grey one $R$, say). Take an edge $e$ of $R$ separating it from the other region, as in part 0 of Fig.~\ref{T3move:case2:fig}. Then:
\begin{figure}[htbp]
\faifig{Quella a pagina 38}
{combinatorial-move-b-3-2}
{The move $T_3$ at the level of dessins d'enfant. Case 2\label{T3move:case2:fig}}
\end{figure}
\begin{enumerate}
\item We add two black and two white vertices on $e$ (part 1);
\item We attach to $(g-1)\cdot T$ a $1$-handle with attaching discs in $R$ (part 2);
\item We collapse along two curly segments as in part 3.\qedhere
\end{enumerate}
\end{proof}

The next move $T_4$ is similar to the previous ones, in particular to $T_3$, but the proof of its validity is more delicate, because in passing
from a realization of $\calD'$ to one of $\calD$ we have to act on vertices of different colours.

\begin{prop}[Move $T_4$]\label{T4:move:prop}
For $g\geqslant1$, $x\geqslant4$, $y\geqslant3$ and $2\leqslant s\leqslant d-2$ the following is a reduction move:
$$\begin{array}{cl}
T_4: & \calD=(\matebold{g}\cdot T,S,\matebold{d},3;\partition{\matebold{x}}*\rho_1,\partition{\matebold{y}}*\rho_2,\partition{\matebold{s},\matebold{d-s}})\\
\leadsto & \calD'=((\bm{g-1})\cdot T,S,\bm{d-2},3;
    \partition{\bm{x-2}}*\rho_1,\partition{\bm{y-2}}*\rho_2,\partition{\bm{s-1},\bm{d-s-1}}).
\end{array}$$
\end{prop}

\begin{proof}
The argument is again similar to those showing Propositions~\ref{T2:move:prop} and~\ref{T3:move:prop},
starting with a dessin d'enfant $\Gamma'$ realizing $\calD'$, with two cases to consider.
But in case 2 we will possibly have to change the given $\Gamma'$
before acting on it, so the explanation is longer.
We call $u$ the black vertex of $\Gamma'$ of valence $x-2$, and $v$ the white one of valence $y-2$.

\medskip

\noindent\textsc{Case 1:}
\emph{$u$ is adjacent to one complementary region of $\Gamma'$ and $v$ is adjacent to the other one}, as in part 0 of Fig.~\ref{T4move:case1:fig}. Then:
\begin{figure}[htbp]
\faifig{Quella a pagina 39 in alto}
{combinatorial-move-b-4-1}
{The move $T_4$ at the level of dessins d'enfant. Case 1\label{T4move:case1:fig}}
\end{figure}
\begin{enumerate}
\item We attach to $(g-1)\cdot T$ a $1$-handle with attaching discs inside the two complementary regions, as in part 1;
\item We add to $\Gamma'$ the co-core of the $1$-handle, in the form of a black and a white vertex joined by two arcs, as in part 2;
\item We collapse along two curly segments as in part 3.
\end{enumerate}

\medskip

\noindent\textsc{Case 2:}
\emph{$u$ and $v$ are completely surrounded by one of the complementary regions of $\Gamma'$}
(the light grey one $R$, say). We claim that up to changing $\Gamma'$ we can realize the following situation:
\emph{There is an edge $e$ shared by the two complementary regions of $\Gamma'$ such that along $\partial R$ (for one of the two possible orientations)
we see (perhaps not consecutively) $u$, then $e$
with its white end first and its black end last, and then $v$}, as in Fig.~\ref{uev:fig}.
\begin{figure}[htbp]
\faifig{Figura come a met\`a di pagina 39, ma togliendo i 2\\ e facendo capire che i due vertici in mezzo sono gli estremi di $e$}{combinatorial-move-b-4-2-boundary-order}{The configuration on $\partial R$ that we aim to realize\label{uev:fig}}
\end{figure}
Note that along $\partial R$ (after we unwind it) we see
$u$ exactly $x-2\geqslant 2$ times and $v$ exactly $y-2\geqslant 1$ times.
If there is a shared edge $e$ such that $u,v,u,e$ appear in this order (perhaps not consecutively)
on $\partial R$, as in Fig.~\ref{uvue:fig} then of course we have the desired configuration.

\begin{figure}[htbp]
\faifig{Figura come la parte a sinistra di quella a pagina 39 in basso della tesi
ma marcando gli estremi di $e$ con dei trattini perpendicolari al bordo del disco
la freccia al centro e la parte destra vanno tagliate}
{combinatorial-move-b-4-2-easy-case}
{A configuration that guarantees that of Fig.~\ref{uev:fig}\label{uvue:fig}}
\end{figure}

Otherwise, we have to modify $\Gamma'$. We proceed as follows, referring to
Fig.~\ref{newgammaprime:fig}.
\begin{figure}[htbp]
\faifig{Quella a pagina 40 della tesi}
{combinatorial-move-b-4-2-2-prelim}
{How to modify $\Gamma'$ to get a configuration as in Fig.~\ref{uvue:fig}\label{newgammaprime:fig}}
\end{figure}

\begin{enumerate}
\item If no $u,v,u,e$ configuration appears along $\partial R$ there is an arc $A$ of $\partial R$
that contains all the occurrences of $u$, and none of $v$ or any shared edge.
Similarly, there is an arc $B$ of $\partial R$ that contains all the occurrences of $v$, none of $u$ and no shared edge.
Note that $u$ and $v$ are not the ends of a shared edge, since they are completely surrounded by $R$.
\item Choose an orientation of $\partial R$ (counterclockwise in the picture),
take  the first occurrence of $u$ in $A$ and let $\alpha$ be the edge immediately after it. Since $\alpha$ is not shared, it appears again
on $\partial R$, with the opposite orientation. Note that $\alpha$ cannot occur again immediately before the first appearance of $u$,
otherwise $u$ would have valence $1$, therefore it will occur somewhere else on $A$.
\item Take the first occurrence of $v$ in $B$ and let $\beta$ be the edge immediately before in $\partial R$.
Since $\beta$ is not shared, it will also occur elsewhere on $\partial R$.
\item Let $a$ be the end of $\alpha$ other than $u$ and $b$ be end of $\beta$ other than $v$.
Erase the edges $\alpha$ and $\beta$ and draw two new ones, connecting $u$ to $v$ and $a$ to $b$ as shown in the picture, getting a new dessin d'enfant $\Gamma'$.
One easily sees that $\Gamma'$ still realizes $\calD'$, and now on the boundary of the
light grey complementary region we see $u,v,u$ in this order, without shared edges in between (because there was none on $B$).
But some shared edge exists, so we get a configuration as in Fig.~\ref{uvue:fig}.
\end{enumerate}

Our claim is proved and we can conclude as follows (see Fig.~\ref{T4move:case2:fig}):
\begin{figure}[htbp]
\faifig{Quella a pagina 41 della tesi}
{combinatorial-move-b-4-2-2}
{The move $T_4$ at the level of dessins d'enfant. Case 2\label{T4move:case2:fig}}
\end{figure}
\begin{enumerate}
\item We add a black and a white vertex on $e$ (part 1);
\item We attach to $(g-1)\cdot T$ a $1$-handle with attaching discs in $R$ (part 2);
\item We collapse along two curly segments as in part 3. \qedhere
\end{enumerate}
\end{proof}

\section{Monodromy and algebraic reduction moves}\label{monodromy:sec}

We recall here the monodromy reformulation of the Hurwitz existence problem and some results from~\cite{EKS}, deducing the existence
of two more reduction moves.
In this section we introduce some notation not used in the rest of the paper
except within the proof of Theorem~\ref{morepoints:thm}.
For $d\geqslant 2$ we denote by $\Pi_d$ the set of partitions of $d$, and for $\pi\in\Pi_d$ we set $v(\pi)=d-\ell(\pi)$, where $\ell(\pi)$ is the length of $\pi$.
Using $v$, the Riemann-Hurwitz condition~(\ref{general:RH}) for a candidate branch datum
$\left(\Theta,\Sigma,d,n;\pi_1,\ldots,\pi_n\right)$
can be written as
$$d\chi(\Sigma)-\chi(\Theta)=v(\pi_1)+\ldots+v(\pi_n),$$
hence for $(g\cdot T,S,d,n;\pi_1,\ldots,\pi_n)$ as
\begin{equation}
\label{v-version:RH}
v(\pi_1)+\ldots+v(\pi_n)=2(d+g-1).
\end{equation}
For $\alpha\in\permu_d$, the group of permutations of $\{1,\ldots,d\}$, we denote by $\pi(\alpha)\in\Pi_d$ the partition of $d$ given by the lengths
of the cycles of $\alpha$ (including the trivial ones of length $1$) and we set $\ell(\alpha)=\ell(\pi(\alpha))$ and $v(\alpha)=v(\pi(\alpha))$.

\paragraph{Review of known results}
We begin by citing a result of~\cite{Hur1891} and deducing one from~\cite{EKS}:

\begin{prop}\label{monodromy:prop}
A candidate branch datum $(g\cdot T,S,d,n;\pi_1,\ldots,\pi_n)$ is realizable if and only if there exist $\theta_1,\ldots,\theta_n\in\permu_d$ such that:
\begin{enumerate}
\item $\pi(\theta_i)=\pi_i$ for $i=1,\ldots,n$;
\item $\theta_1\cdots\theta_n=\operatorname{id}$;
\item The subgroup $\left\langle\theta_1,\ldots,\theta_n\right\rangle$ of $\permu_d$ acts transitively on $\{1,\ldots,d\}$.
\end{enumerate}
\end{prop}

\begin{prop}\label{smallv:existence:prop}
If $\pi_1,\pi_2\in\Pi_d$ and $v(\pi_1)+v(\pi_2)\leqslant d-1$
there exist $\theta_1,\theta_2\in\permu_d$ such that $\pi(\theta_i)=\pi_i$ for $i=1,2$ and
$v(\theta_1\cdot\theta_2)=v(\pi_1)+v(\pi_2)$.
\end{prop}

\begin{proof}
According to Lemma~4.2 in~\cite{EKS}, if $v(\pi_1)+v(\pi_2)=d-t$ for some $t\geqslant 1$
then there exist $\theta_1,\theta_2\in\permu_d$ such that $\pi(\theta_i)=\pi_i$ for $i=1,2$,
the subgroup $\left\langle\theta_1,\theta_2\right\rangle$ of $\permu_d$ has precisely $t$ orbits,
and $\pi(\theta_1\cdot\theta_2)$ is the partition of $d$ given by the lengths of these orbits.
This implies that $v(\theta_1\cdot\theta_2)=d-t$, so
$v(\theta_1\cdot\theta_2)=v(\pi_1)+v(\pi_2)$.
\end{proof}

We then reformulate Corollary~4.4 and Lemma~4.5 from~\cite{EKS}, respectively:

\begin{prop}\label{largev:truecycle:existence:prop}
Let $\pi_1,\pi_2\in\Pi_d$ be such that $v(\pi_1)+v(\pi_2)\geqslant d-1$ and $v(\pi_1)+v(\pi_2)\equiv d-1\pmod{2}$.
Then there exist $\theta_1,\theta_2\in\permu_d$ such that $\pi(\theta_i)=\pi_i$ for $i=1,2$ and
$\pi(\theta_1\cdot\theta_2)=\partition{d}$.
\end{prop}

\begin{prop}\label{largev:twocycles:existence:prop}
Let $\pi_1,\pi_2\in\Pi_d$ be such that $v(\pi_1)+v(\pi_2)\geqslant d$ and $v(\pi_1)+v(\pi_2)\equiv d\pmod{2}$.
Then there exist $\theta_1,\theta_2\in\permu_d$ such that $\pi(\theta_i)=\pi_i$ for $i=1,2$ and
$$
\pi(\theta_1\cdot\theta_2)=\begin{cases}
\partition{d/2,d/2} &\hbox{\ if\ \ \ }\pi_1=\pi_2=\partition{2,\ldots,2}\\
\partition{d-1,1}   &\hbox{\ otherwise}.
\end{cases}
$$
\end{prop}

\paragraph{Algebraic reduction moves}
We introduce here two moves $A_1$ and $A_2$ used in Section~\ref{morepoints:sec}.

\begin{prop}[Move $A_1$]\label{A1:move:prop}
If $\pi_1,\pi_2\in\Pi_d$ and $v(\pi_1)+v(\pi_2)\leqslant d-1,$ then for all
$\theta_1,\theta_2\in\permu_d$ such that $\pi(\theta_i)=\pi_i$ for $i=1,2$ and
$v(\theta_1\cdot\theta_2)=v(\pi_1)+v(\pi_2)$,
setting $\pi=\pi(\theta_1\cdot\theta_2)$ we have that the following is a reduction move:
$$\begin{array}{cl}
A_1: & \calD=(g\cdot T,S,d,\bm{n};\bm{\pi_1},\bm{\pi_2},\pi_3,\ldots,\pi_n)\\
\leadsto & \calD'=(g\cdot T,S,d,\bm{n-1};\bm{\pi},\pi_3,\ldots,\pi_n).
\end{array}$$
\end{prop}

\begin{proof}
Version~(\ref{v-version:RH}) of the Riemann-Hurwitz relation implies that if $\calD$ as in the statement
is a candidate branch datum, $\calD'$ also is
(since $v(\pi)=v(\pi_1)+v(\pi_2)>0$ we see that $\pi$ is non-trivial). Now if $\theta,\theta_3,\ldots,\theta_n$ realize $\calD'$ according to Proposition~\ref{monodromy:prop},
we can assume that $\theta=\theta_1\cdot\theta_2$, whence $\theta_1,\theta_2,\theta_3,\ldots,\theta_n$ realize $\calD$.
\end{proof}

\begin{prop}[Move $A_2$]\label{A2:move:prop}
Given $d\geq3$ and $\pi_1,\ldots,\pi_n\in\Pi_d$ with $v(\pi_1)+v(\pi_2)\geqslant d-1$ and
$v(\pi_3)+\ldots+v(\pi_n)\geqslant d-1$, if $g=\frac12(v(\pi_1)+\ldots+v(\pi_n))-d+1$ is a non-negative integer
then there exists $g'\in\matN$ and $\pi\in\Pi_d$ such that
the following is a reduction move:
$$\begin{array}{cl}
A_2: & \calD=(\bm{g}\cdot T,S,d,\bm{n};\bm{\pi_1},\bm{\pi_2},\pi_3,\ldots,\pi_n)\\
\leadsto & \calD'=(\bm{g'}\cdot T,S,d,\bm{n-1};\bm{\pi},\pi_3,\ldots,\pi_n),
\end{array}$$
where $\pi=\partition{d/2,d/2}$ if $\pi_1=\pi_2=\partition{2,\ldots,2}$, and
$$
\pi=\begin{cases}
\partition{d}     &\hbox{\ if \ \ }v(\pi_1)+v(\pi_2)\equiv d-1\pmod{2}\\
\partition{d-1,1} &\hbox{\ if \ \ }v(\pi_1)+v(\pi_2)\equiv d\pmod{2}
\end{cases}
$$
otherwise.
\end{prop}

\begin{proof}
Set $v_j=v(\pi_j)$. If $v_1+v_2\equiv d-1\pmod{2}$ we can apply Proposition~\ref{largev:truecycle:existence:prop}
getting $\theta_1,\theta_2\in\permu_d$ such that $\pi(\theta_i)=\pi_i$ for $i=1,2$ and
$\pi(\theta_1\cdot\theta_2)=\partition{d}$.
If instead $v_1+v_2\equiv d\pmod{2}$ we note that $v_1+v_2\geqslant d-1$ implies $v_1+v_2\geqslant d$, so we can apply
Proposition~\ref{largev:twocycles:existence:prop}, getting
$\theta_1,\theta_2\in\permu_d$ such that $\pi(\theta_i)=\pi_i$ for $i=1,2$ and
$$
\pi(\theta_1\cdot\theta_2)=\begin{cases}
\partition{d/2,d/2} &\hbox{\ if\ \ \ }\pi_1=\pi_2=\partition{2,\ldots,2}\\
\partition{d-1,1}   &\hbox{\ otherwise}.
\end{cases}
$$

Setting $\pi=\pi(\theta_1\cdot\theta_2)$ we claim that in both cases if $\calD$ as in the statement
is a candidate branch datum, so is $\calD'$ for a suitable $g'\in\matN$. Once this is done, the conclusion is precisely as in the previous proof.
First of all, $\pi$ is always non-trivial. Next, we must show that~(\ref{v-version:RH}) holds for $\calD'$, namely that
$$v(\pi)+v_3+\ldots+v_n=2(d+g'-1)$$
holds for some $g'\in\matN$, or equivalently that
$$z=v(\pi)+v_3+\ldots+v_n-2d+2$$
is even and non-negative. Now in each of the three cases one readily sees that $v(\pi)\equiv v_1+v_2\pmod{2}$, but
$v_1+v_2+v_3+\ldots+v_n$ is even by~(\ref{v-version:RH}) for $\calD$, so indeed $z$ is even.
Moreover $v(\pi)\geqslant d-2$ and $v_3+\ldots+v_n\geqslant d-1$, so $z\geqslant-1$ and the conclusion follows.
\end{proof}

\section{Statement and exceptionality}\label{statement:sec}

We now state the result to which the present paper is devoted:

\begin{thm}\label{main:thm}
A candidate branch datum
$(g\cdot T,S,d,n;\pi_1,\ldots,\pi_n)$ with $\ell(\pi_n)=2$ is exceptional if and only if it is one of the following:
\begin{enumerate}
\item[(1)] $(S,S,12,3;\partition{2,2,2,2,2,2},\partition{1,1,1,3,3,3},\partition{6,6})$;
\item[(2)] $(S,S,2k,3;\partition{2,\ldots,2},\partition{2,\ldots,2},\partition{s,2k-s})$ with $k\geqslant 2$ and $s\neq k$;
\item[(3)] $(S,S,2k,3;\partition{2,\ldots,2},\partition{1,2,\ldots,2,3},\partition{k,k})$ with $k\geqslant 2$;
\item[(4)] $(S,S,4k+2,3;\partition{2,\ldots,2},\partition{1,\ldots,1,k+1,k+2},\partition{2k+1,2k+1})$ with $k\geqslant 1$;
\item[(5)] $(S,S,4k,3;\partition{2,\ldots,2},\partition{1,\ldots,1,k+1,k+1},\partition{2k-1,2k+1})$ with $k\geqslant 2$;
\item[(6)] $(S,S,kh,3;\partition{h,\ldots,h},\partition{1,\ldots,1,k+1},\partition{ph,(k-p)h})$ with $h\geqslant 2$, $k\geqslant 2$ and $0<p<k$;
\item[(7)] $(T,S,6,3;\partition{3,3},\partition{3,3},\partition{2,4})$;
\item[(8)] $(T,S,8,3;\partition{2,2,2,2},\partition{4,4},\partition{3,5})$;
\item[(9)] $(T,S,12,3;\partition{2,2,2,2,2,2},\partition{3,3,3,3},\partition{5,7})$;
\item[(10)] $(T,S,16,3;\partition{2,2,2,2,2,2,2,2},\partition{1,3,3,3,3,3},\partition{8,8})$;
\item[(11)] $(T,S,2k,3;\partition{2,\ldots,2},\partition{2,\ldots,2,3,5},\partition{k,k})$ with $k\geqslant 5$;
\item[(12)] $(2\cdot T,S,8,4;\partition{2,2,2,2},\partition{2,2,2,2},\partition{2,2,2,2},\partition{3,5})$;
\item[(13)] $((n-3)\cdot T,S,4,n;\partition{2,2},\ldots,\partition{2,2},\partition{1,3})$ with $n\geqslant3$.
\end{enumerate}
\end{thm}

\begin{rem}\emph{By Remark~\ref{nontrivial:rem} we always assume from now on that $n\geqslant 3$.}
\end{rem}

\begin{rem}\emph{For the case where the candidate covering surface is the sphere $S$, Pakovich~\cite{Pako} lists $7$ exceptional families rather than the $6$ above.
Our statement is however coherent with his one, because two of his families  are actually equivalent to each other.
In fact, items (4) and (5) in his statement, using our current notation, are respectively
\begin{itemize}
  \item[(4)] $(S,S,d,3;\partition{2,\ldots,2},\partition{1,\ldots,1,q,q},\partition{2q-3,d-2q+3})$ with $q\geqslant3$;
  \item[(5)] $(S,S,d,3;\partition{2,\ldots,2},\partition{1,\ldots,1,q,q},\partition{2q-1,d-2q+1})$ with $q\geqslant3$.
\end{itemize}
For both of them the Riemann-Hurwitz condition~(\ref{general:RH}) reads
$$2-\left(\frac d2+(d-2q+2)+2\right)=d(2-3)\quad\Rightarrow\quad d=4q-4.$$
So in (4) we have $d-2q+3=2q-1$ and in (5) we have $d-2q+1=2q-3$, whence (4) and (5) are actually the same.
Setting $k=q-1$ we have
$$k\geqslant2,\quad\ q=k+1,\quad d=4k,\quad 2q-3=2k-1,\quad
2q-1=2k+1$$
so both candidate branch data are encoded as our (5).}
\end{rem}

\begin{rem}\emph{The items in our statement have a few overlaps. For instance item (2) with $k=2$ coincides with item (6) with $k=h=2$.
Easy extra restrictions on the parameters appearing would lead to a list without overlaps, but we will not make this explicit.}
\end{rem}

In the rest of the section we explain why items~(1) to~(13) in Theorem~\ref{main:thm} are indeed exceptional, referring to the existing literature and in some cases also 
providing direct proofs.

\paragraph{Exceptionality with covering surface the sphere}
We begin with items~(1) to~(6) in Theorem~\ref{main:thm}.
Their exceptionality was proved in~\cite{Pako} within the general delicate argument showing the
realizability of the candidate branch data not falling in these families, but we show here that it is also possible to establish it in a more elementary fashion.

For item~(1), one could use the data of~\cite{Zh06}, but a proof via dessins d'enfant is also very easy:

\begin{prop}
The candidate branch datum (1) in Theorem~\ref{main:thm} is exceptional.
\end{prop}

\begin{proof}
A dessin d'enfant $\Gamma$ realizing (1),
ignoring the 2-valent black vertices,
would be a connected graph in $S$ with three $3$-valent and three $1$-valent white vertices, and $6$ edges.
Then $\Gamma$ is obtained from the graph of Fig.~\ref{1exceptional:fig}
\begin{figure}
\begin{center}
\begin{picture}(80,30)
\put(10,0){\line(0,1){13.5}}
\put(10,30){\line(0,-1){13.5}}
\put(40,30){\line(0,-1){13.5}}
\put(70,0){\line(0,1){13.5}}
\put(70,30){\line(0,-1){13.5}}
\put(11.5,15){\line(1,0){27}}
\put(41.5,15){\line(1,0){27}}
\put(-1,0){$A$}
\put(-1,25){$B$}
\put(28,25){$C$}
\put(72,25){$D$}
\put(72,0){$E$}
\put(10,15){\circle{3}}
\put(40,15){\circle{3}}
\put(70,15){\circle{3}}
\end{picture}
\end{center}
\mycap{A graph in $S$\label{1exceptional:fig}}
\end{figure}
by joining two of its free ends $A,B,C,D,E$ and putting a valence-$1$ vertex at the end of the other three.
Of the 10 possible junctions, up to symmetry we can consider only
$A$-$B$, $A$-$C$, $A$-$D$, $A$-$E$, $B$-$C$, $B$-$D$, that give for the lengths of the complementary regions respectively
$\partition{1,11}$, $\partition{4,8}$, $\partition{5,7}$, $\partition{3,9}$, $\partition{2,10}$, $\partition{5,7}$, so $\partition{6,6}$ does not appear.
\end{proof}

Item~(2) is exceptional because a dessin d'enfant realizing it would be a circle with $2k$ valence-$2$ vertices on it (of alternating colours), but
then its complementary regions would have lengths $\partition{k,k}$.
Exceptionality of item~(3) was shown in~\cite[Proposition 1.3]{PePe06}. That of items~(4) to~(6) perhaps follows from some
published result other than those in~\cite{Pako}, 
but showing it via dessins d'enfant is quite easy, so we do it.
We note that item~(6) for the special case $p=1$ is treated in~\cite[Proposition~5.7]{EKS}.

\begin{prop}
Items~(4) to~(6) in Theorem~\ref{main:thm} are exceptional.
\end{prop}

\begin{proof}
We will make the argument explicit for the hardest case~(6), 
leaving the easier~(4) and~(5) to the reader.

A dessin d'enfant $\Gamma\subset S$ realizing the (6), being connected, must appear as in Fig.~\ref{kh:exception:fig}
\begin{figure}
\include{Figs/preamble-exceptionality.tex}
\faifig{Quella a pagina 55 della tesi aggiungendo due volte a sinistra le graffe con $h-1$}
{exceptionality-composite}
{Proof of the exceptionality of item (6)\label{kh:exception:fig}}
\end{figure}
for some $0\leqslant a\leqslant k-1$ and $0\leqslant b\leqslant h-2$.
But then one sees that one of the complementary regions has length
$ah+b+1$ which cannot be $ph$ or $(k-p)h$ because it is not a multiple of $h$.
\end{proof}

\paragraph{Exceptionality with covering surface of positive genus}
As above for item~(1), items~(7) to~(10) in Theorem~\ref{main:thm} are very easily shown to be exceptional via dessins d'enfant,
but we do not exhibit a proof as all these cases fall within the experimental analysis of Zheng~\cite{Zh06}.
Exceptionality of item (11) was shown in~\cite[Proposition~1.2]{PePe06}.
We then turn to~(12), to which the technique of dessins d'enfant does not apply,
as there are four rather than three candidate branching points in $S$.

\begin{prop}
The candidate branch datum~(12) in Theorem~\ref{main:thm} is exceptional.
\end{prop}

\begin{proof}
Again we could just apply the method of Zheng~\cite{Zh06}, but we sketch an alternative argument that exploits a graphic technique not used elsewhere in the present paper or
in~\cite{Pako}. In fact, it extends  dessins d'enfant but in a different way than
the constellations used in~\cite{Pako}.

\medskip

Suppose a branched covering $f$ realizing item~(12) exists, and let $e$ be a segment with ends at the first two branching points,
  midpoint at the third one and avoiding the fourth one. Then $f^{-1}(e)$ is a graph $\Gamma$ in $2\cdot T$ with four $4$-valent vertices and two complementary discs incident
  to respectively $6$ and $10$ vertices. The fact that no such $\Gamma$ exists is shown as follows:
\begin{itemize}
\item Note that abstractly such a $\Gamma$ always has as a maximal tree $\Lambda$ as in Fig.~\ref{44tree:fig}
\begin{figure}
\begin{center}
\begin{picture}(150,15)
\put(0,15){\line(1,0){150}}
\put(-4,2){$A$}
\put(46,2){$B$}
\put(96,2){$C$}
\put(146,2){$D$}
\put(0,15){\circle*{3}}
\put(50,15){\circle*{3}}
\put(100,15){\circle*{3}}
\put(150,15){\circle*{3}}
\end{picture}
\end{center}
\mycap{A tree $\Lambda$\label{44tree:fig}}
\end{figure}
(if it has none then there is one that is a spider with a head $A$ and three legs with ends $B,C,D$, but $B$ cannot be joined to $C$ or $D$,
hence $C$ is joined to $D$, a contradiction);
\item List all the ways the other 5 edges of $\Gamma$ can be abstractly attached to $\Lambda$ and eliminate duplicates up to symmetry,
    concluding that there are $10$ possibilities for the abstract $\Gamma$ (we do not show them explicitly);
\item Note that if $\Gamma$ is embedded in a surface $\Sigma$ then a neighbourhood $U$ of
$\Lambda$ in $\Sigma$ is contained in a plane and it is one of the four shown in Fig.~\ref{44tree-neigh:fig};
\begin{figure}
\begin{center}
\begin{picture}(150,30)
\put(-15,15){\line(1,0){180}}
\put(0,0){\line(0,1){30}}
\put(50,0){\line(0,1){30}}
\put(100,0){\line(0,1){30}}
\put(150,0){\line(0,1){30}}
\put(-12,2){$A$}
\put(38,2){$B$}
\put(88,2){$C$}
\put(138,2){$D$}
\put(0,15){\circle*{3}}
\put(50,15){\circle*{3}}
\put(100,15){\circle*{3}}
\put(150,15){\circle*{3}}
\end{picture}

\vspace{.5cm}

\begin{picture}(150,30)
\put(-15,15){\line(1,0){180}}
\put(0,0){\line(0,1){30}}
\put(50,0){\line(0,1){30}}
\put(100,15){\line(-1,1){15}}
\put(100,15){\line(1,1){15}}
\put(150,0){\line(0,1){30}}
\put(-12,2){$A$}
\put(38,2){$B$}
\put(96,2){$C$}
\put(138,2){$D$}
\put(0,15){\circle*{3}}
\put(50,15){\circle*{3}}
\put(100,15){\circle*{3}}
\put(150,15){\circle*{3}}
\end{picture}

\vspace{.5cm}

\begin{picture}(150,30)
\put(-15,15){\line(1,0){180}}
\put(0,0){\line(0,1){30}}
\put(50,15){\line(-1,1){15}}
\put(50,15){\line(1,1){15}}
\put(100,15){\line(-1,1){15}}
\put(100,15){\line(1,1){15}}
\put(-12,2){$A$}
\put(46,2){$B$}
\put(96,2){$C$}
\put(138,2){$D$}
\put(0,15){\circle*{3}}
\put(50,15){\circle*{3}}
\put(100,15){\circle*{3}}
\put(150,15){\circle*{3}}
\end{picture}

\vspace{.5cm}

\begin{picture}(150,30)
\put(-15,15){\line(1,0){180}}
\put(0,0){\line(0,1){30}}
\put(50,15){\line(-1,-1){15}}
\put(50,15){\line(1,-1){15}}
\put(100,15){\line(-1,1){15}}
\put(100,15){\line(1,1){15}}
\put(150,0){\line(0,1){30}}
\put(-12,2){$A$}
\put(46,19){$B$}
\put(96,2){$C$}
\put(138,2){$D$}
\put(0,15){\circle*{3}}
\put(50,15){\circle*{3}}
\put(100,15){\circle*{3}}
\put(150,15){\circle*{3}}
\end{picture}
\end{center}
\mycap{Possibilities for the planar neighbourhood $U$ of $\Lambda$ in $\Sigma$\label{44tree-neigh:fig}}
\end{figure}
\item Fix one of the $10$ abstract $\Gamma$'s and one of the $U$'s of Fig.~\ref{44tree-neigh:fig};
now the compatible $\Gamma$'s in $\Sigma$ are those obtained by:
\begin{itemize}
\item Joining in pairs, by arcs in the plane that may cross each other, the $10$ germs of edges of $\Lambda$, so to get $\Gamma$;
\item Taking the circles bounding a neighbourhood of $\Gamma$ in the plane (ignoring the crossings);
\item Attaching a disc to each such circle;
\item Computing to how many vertices these discs are incident;
\end{itemize}
\item This description suggests that several cases have to be considered for each of the $10\times 4$ possibilities. But
our aim is just to show that we never find two attaching discs incident to respectively $6$ and $10$ vertices, so we can
discard a partially constructed $\Gamma$ as soon as we see it creates an attaching disc incident to some number of vertices different from $6$ and $10$.
The resulting analysis is then rather simple and quick, leading to the desired conclusion.
\end{itemize}
\end{proof}

Finally, item~(13) was shown to be exceptional
in~\cite{EKS}, exploiting the easy fact that the identity and the elements $\sigma$ of $\permu_4$ with $\pi(\sigma)=\partition{2,2}$ form a subgroup of $\permu_4$.


\section{Realizability for three branching points}\label{threepoints:sec}
In this section we prove Theorem~\ref{main:thm} under the restriction that the number $n$ of branching points
of the candidate branch datum is $3$. As announced, we proceed by induction on the genus $g$ of
the candidate covering surface, with the base step $g=0$ being a consequence of~\cite{Pako}:

\begin{thm}\label{sphere:thm}
The only exceptional candidate branch data of the form
$$\calD=(S,S,d,3;\pi_1,\pi_2,\partition{s,d-s})$$
are items~(1) to~(6) in Theorem~\ref{main:thm}.
\end{thm}

The inductive step will use the reduction moves $T_1,\ldots,T_4$ of Section~\ref{dessins:sec}.

\paragraph{Sparse realizability results}
For some specific candidate branch data we will not be able to apply any reduction move, so we treat them here.
We begin by stating a fact which is easily deduced from~\cite[Theorem~1.2]{PePe08},
and then we employ dessins d'enfant to establish another fact.

\begin{prop}\label{xy22:prop}
For $x\geqslant 3$ and $y\geqslant 2$ any candidate branch datum as follows is realizable:
$$\calD=(g\cdot T,S,d,3;\partition{x,y}*\rho_1,\partition{2,2}*\rho_2,\partition{2,d-2}).$$
\end{prop}

\begin{prop}\label{dessins:special-families:prop}
For $g\geqslant 2$ the following candidate branch data are realizable:
\begin{enumerate}
\item $(g\cdot T,S,6g,3;\partition{3,\ldots,3},\partition{3,\ldots,3},\partition{s,6g-s})$;
\item $(g\cdot T,S,6g+2,3;\partition{2,3,\ldots,3},\partition{2,3\ldots,3},\partition{s,6g+2-s})$;
\item $(g\cdot T,S,6g+3,3;\partition{3,\ldots,3},\partition{1,2,3,\ldots,3},\partition{s,6g+3-s})$;
\item $(g\cdot T,S,6g+4,3;\partition{1,3,\ldots,3},\partition{1,3,\ldots,3},\partition{s,6g+4-s})$;
\item $(g\cdot T,S,6g+6,3;\partition{1,2,3,\ldots,3},\partition{1,2,3,\ldots,3},\partition{s,6g+6-s})$.
\end{enumerate}
\end{prop}

\begin{proof}
Within this proof we define an \emph{augmented datum} as a symbol $\calD=(g\cdot T,S,d,3;\pi_1,\pi_2,\pi_3)$
similar to a candidate branch datum, except that some entries of $\pi_3$ are underlined.
We say that $\calD$ is \emph{realizable} if there exists a dessin d'enfant $\Gamma$ realizing $\calD$ (forgetting the underlining) such that,
for every complementary region $R$ of $\Gamma$ corresponding to an underlined entry of $\pi_3$, there is an edge of $\Gamma$ with $R$ on both sides.

\medskip

\noindent\textsc{Fact 1.}
\emph{For $\pi_3\in\{\partition{1,\underline{11}},\partition{2,\underline{10}},\partition{\underline{3},\underline{9}},\partition{\underline{4},\underline{8}},\partition{\underline{5},\underline{7}},\partition{\underline{6},\underline{6}}\}$
the augmented datum
$\calD=(2\cdot T,S,12,3;\partition{3,3,3,3},\partition{3,3,3,3},\pi_3)$ is realizable.}
This is established by exhibiting the desired dessins d'enfant in Fig.~\ref{degree12:augmented:fig}.
\begin{figure}
\faifig{Quella a pagina 43 in alto della tesi ma togliendo tutti i colori}
{dessin-special-families-small}
{Fact 1 at the level of dessins d'enfant.
These pictures only show an embedding in $\matR^3$ of a regular neighbourhood of a dessin $\Gamma$ in $2\cdot T$
\label{degree12:augmented:fig}}
\end{figure}

\medskip

\noindent\textsc{Fact 2.}
\emph{If $\calD=(g\cdot T,d,\pi_1,\pi_2,\partition{\underline{x},d-x})$ is realizable then
$$\calD'=((g+1)\cdot T,d+6,\partition{3,3}*\pi_1,\partition{3,3}*\pi_2,\partition{\underline{x+6},d-x})$$ also is.}
To see this, take a dessin d'enfant $\Gamma$ realizing $\calD$, with light grey complementary region $R$ corresponding to $\underline{x}$,
and fix an edge $e$ of $\Gamma$ with $R$ on both sides, as in part 0 of
Fig.~\ref{augmented:6jump:fig}. We then operate as follows on $\Gamma$ to get a $\Gamma'$ realizing $\calD'$:
\begin{figure}
\faifig{Quella a pagina 43 in basso e 44 in alto, se possibile\\ alzando il cocuore del manico sopra il disco}{dessin-special-families-cmove}{Fact 2 at the level of dessins d'enfant\label{augmented:6jump:fig}}
\end{figure}
\begin{enumerate}
\item We add a black and a white vertex on $e$ (part 1);
\item We attach to $g\cdot T$ a $1$-handle with attaching discs inside $R$ (part 2);
\item We add one black vertex, one white vertex and four edges, as in part 3.
\end{enumerate}

\medskip

\noindent\textsc{Fact 3.}
\emph{For $g\geqslant 2$ and $\pi_3\in\{\partition{1,\underline{6g-1}},\partition{2,\underline{6g-2}}\}\cup\{\partition{\underline{s},\underline{6g-s}}\colon 3\leqslant s\leqslant 6g-3\}$ the augmented datum
$(g\cdot T,S,6g,3;\partition{3,\ldots,3},\partition{3,\ldots,3},\pi_3)$
is realizable}.
This is easily shown by induction on $g$, using Fact 1 for the base and Fact 2 for the induction.

\medskip

\noindent\textsc{Fact 4.}
\emph{If
$\calD=(g\cdot T,S,d,3;\pi_1,\pi_2,\partition{\underline{x},d-x})$ is realizable then
$$\calD'=(g\cdot T,S,d+2,3;\partition{2}*\pi_1,\partition{2}*\pi_2,\partition{\underline{x+2},d-x})$$ also is.}
To see this it is enough to take a dessin d'enfant $\Gamma$ realizing $\calD$ and to add one black vertex and one white vertex on
an edge $e$ having on both sides the complementary region of $\Gamma$ corresponding to $\underline{x}$, as in Fig.~\ref{augmented:2jump:fig}.
\begin{figure}
\faifig{Quella a pagina 44 in mezzo}
{dessin-special-families-step-4}
{Fact 4 at the level of dessins d'enfant\label{augmented:2jump:fig}}
\end{figure}

\medskip

\noindent\textsc{Fact 5.}
\emph{If
$\calD=(g\cdot T,S,d,3;\pi_1,\pi_2,\partition{\underline{x},d-x})$ is realizable then
$$\calD'=(g\cdot T,S,d+3,3;\partition{3}*\pi_1,\partition{1,2}*\pi_2,\partition{\underline{x+3},d-x})$$ also is.}
In the usual framework, proceed as in Fig.~\ref{augmented:3jump:fig}.
\begin{figure}
\faifig{Quella a pagina 44 in basso}
{dessin-special-families-step-5}
{Fact 5 at the level of dessins d'enfant\label{augmented:3jump:fig}}
\end{figure}

\medskip

\noindent\textsc{Fact 6.}
\emph{If
$\calD=(g\cdot T,S,d,3;\pi_1,\pi_2,\partition{\underline{x},d-x})$ is realizable then
$$\calD'=(g\cdot T,S,d+4,3;\partition{1,3}*\pi_1,\partition{1,3}*\pi_2,\partition{\underline{x+4},d-x})$$ also is.}
See Fig.~\ref{augmented:4jump:fig}.
\begin{figure}
\faifig{Quella a pagina 45 in basso}
{dessin-special-families-step-6}
{Fact 6 at the level of dessins d'enfant\label{augmented:4jump:fig}}
\end{figure}

\medskip

\noindent\textsc{Conclusion.}
It is now easy to see that each of the five candidate branch data of the statement can be obtained starting from the realizable augmented datum of Fact 3
by applying Facts 4, 5 and 6 zero or more times, and then forgetting about the augmentation. Namely:
\begin{enumerate}
\item No need to apply steps 4, 5 or 6;
\item Apply Fact 4;
\item Apply Fact 5;
\item Apply Fact 6;
\item Apply Facts 4 and 6.\qedhere
\end{enumerate}
\end{proof}

\paragraph{Genus-1 covering surface}
We now prove the following:

\begin{thm}\label{torus:thm}
The only exceptional candidate branch data of the form
$$\calD=(T,S,d,3;\pi_1,\pi_2,\partition{s,d-s})$$
are items~(7) to~(11) in Theorem~\ref{main:thm}.
\end{thm}

\begin{proof}
Setting $\ell_j=\ell(\pi_j)$ we see that the Riemann-Hurwitz condition~(\ref{general:RH}) reads $\ell_1+\ell_2=d-2$.
First note that~\cite[Proposition~5.3]{EKS} implies that no datum of the form
$$(T,S,d,n;\pi_1,\ldots,\pi_{n-1},\partition{1,d-1})$$
is exceptional, so we can assume $1<s<d-1$. Moreover, according to Table 3 in~\cite{Zh06}, the
only relevant exceptions with $d\leqslant16$ are items~(7) to~(10) in Theorem~\ref{main:thm} and item~(11) for $5\leqslant k\leqslant8$, so we can assume $d\geqslant17$
and we are left to show that only item~(11) is exceptional.

The rest of the proof is split in the analysis of various cases.

\medskip

\noindent
\textsc{Case 1:}
$\pi_1=\partition{\matebold{x}}*\rho_1,\ \pi_2=\partition{\mathbf{2}}*\rho_2$ with $x\geqslant4$ and $\rho_2\neq \partition{2,\ldots,2}$.
We apply to $\calD$ the reduction move $T_2$ with $x_1=1$ and $x_2=x-3$ getting
$$\calD'=(S,S,d-2,3;\partition{1,x-3}*\rho_1,\rho_2,\partition{s-1,d-s-1})$$
(here and in the sequel when describing a case we already highlight the entries of $\pi_1$ and $\pi_2$ at which we will later apply a move).
Since neither $\partition{1,x-3}*\rho_1$ nor $\rho_2$ can be $\partition{2,\ldots,2}$, we see that $\calD'$ is realizable by Theorem~\ref{sphere:thm} unless it is item~(6) with $d-2=kh,\ \rho_2=\partition{h,\ldots,h}$ (so $h\geqslant3$), $s-1=ph$ and either $\rho_1=\partition{1,\ldots,1}$ and $x-3=k+1$ or $x-3=1$ and $\rho_1=\partition{1,\ldots,1,k+1}$.
Correspondingly, we have that $\calD$ is one of the following:
$$\begin{array}{c}
(T,S,kh+2,3;\partition{1,\ldots,1,\bm{k+4}},\partition{\mathbf{2},h,\ldots,h},\partition{ph+1,(k-p)h+1})\\
(T,S,kh+2,3;\partition{1,\ldots,1,\mathbf{4},k+1},\partition{2,\matebold{h},h,\ldots,h},\partition{ph+1,(k-p)h+1}).
\end{array}$$
If we apply respectively the moves $T_2$ and $T_4$ at the highlighted entries of $\pi_1$ and $\pi_2$, we get
$$\begin{array}{c}
(S,S,kh,3;\partition{1,\ldots,1,2,k},\partition{h,\ldots,h},\partition{ph,(k-p)h})\\
(S,S,kh,3;\partition{1,\ldots,1,2,k+1},\partition{2,h-2,h,\ldots,h},\partition{ph,(k-p)h})
\end{array}$$
that are realizable by Theorem~\ref{sphere:thm}.

\medskip

\noindent
\textsc{Case 2:}
$\pi_1=\partition{\matebold{x}}*\rho_1, \pi_2=\partition{\matebold{y}}*\rho_2$ with $x,y\geqslant4$, $\rho_1\not\ni2,\ \rho_2\not\ni2$.
Then we apply $T_4$ getting
$$\calD'=(S,S,d-2,3;\partition{x-2}*\rho_1,\partition{y-2}*\rho_2,\partition{s-1,d-s-1})$$
which is realizable by Theorem~\ref{sphere:thm} unless it is item~(6) with (up to switch)
$d-2=kh,\ \rho_1=\partition{1,\ldots,1},\ x-2=k+1,\ y-2=h$, $\rho_2=\partition{h,\ldots,h}$, whence $h\geqslant3$, and $s-1=ph$, so $\calD$ is
$$(T,S,kh+2,3;\partition{1,\ldots,1,\bm{k+3}},\partition{h+2,\matebold{h},h,\ldots,h},\partition{ph+1,(k-p)h+1})$$
and applying $T_4$ we get the following datum which is realizable by Theorem~\ref{sphere:thm}:
$$(S,S,kh,3;\partition{1,\ldots,1,k+1},\partition{h+2,h-2,h,\ldots,h},\partition{ph,(k-p)h}).$$

\medskip

\noindent
\textsc{Case 3:}
$\pi_1=\partition{\matebold{x}}*\rho_1$ with $x\geqslant4$, $\pi_2=\partition{\mathbf{3}}*\rho_2$ with $\rho_2$ containing $1$'s and $3$'s only.
Then we apply $T_4$ getting
$$(S,S,d-2,3;\partition{x-2}*\rho_1,\partition{1}*\rho_2,\partition{s-1,d-s-1})$$
which by Theorem~\ref{sphere:thm} is realizable unless it is item~(6) with $d-2=kh$, $x-2=h$, $\rho_1=\partition{h,\ldots,h}$, $\rho_2=\partition{1,\ldots,1,k+1}$, whence $k=2$,
and $s-1=ph$. Then we have $p=1$ and $d=2h+2$, $h\geqslant8$ and
$$\calD=(T,S,2h+2,3;\partition{h+2,h},\partition{3,\mathbf{3},1,1,\ldots,1},\partition{h+1,h+1})$$
but applying $T_1$ we get the following datum which is realizable by Theorem~\ref{sphere:thm}:
$$(S,S,2h+2,3;\partition{h+2,h},\partition{3,1,\ldots,1},\partition{h+1,h+1}).$$

\medskip

\noindent
\textsc{Case 4:}
$\max(\pi_1)=3,\ \max(\pi_2)\leqslant3$, $\pi_2\neq\partition{2,\ldots,2}$.
We first note that if $\pi_1\supseteq\partition{1,1}$ we have $\pi_1=\partition{\mathbf{3},1,1}*\rho_1$ and we can apply $T_1$ getting
$$(S,S,d,3;\partition{1,1,1,1,1}*\rho_1,\pi_2,\partition{s,d-s})$$
which can be exceptional only if it is item~(6) with $k=2$ and $h\leqslant 3$, hence $d\leqslant6$, which we are excluding.
So we assume $\pi_1\not\supseteq\partition{1,1}$ and we face the following subcases:
\begin{itemize}
\item[(i)] $\pi_1=\partition{\mathbf{3},\mathbf{3}}*\rho_1,\ \pi_2=\partition{\mathbf{2},\mathbf{2}}*\rho_2$;
\item[(ii)] $\pi_2\not\supseteq\partition{2,2}$;
\item[(iii)] $\pi_1\not\supseteq\partition{3,3}$.
\end{itemize}
In subcase~(i) for $s=2$ Proposition~\ref{xy22:prop} implies that $\calD$ is realizable. Then we can assume $3\leqslant s\leqslant d-3$ and apply $T_3$, getting
$$(S,S,d-4,3;\partition{1,1}*\rho_1,\rho_2,\partition{s-2,d-s-2})$$
that can be exceptional only if it is item~(6) with $k+1\leqslant 3$, $h\leqslant 3$ and $kh=d-4$, whence $d\leqslant10$, which we are excluding.
In subcase~(ii) we first note that $\pi_2$ contains $3$'s, otherwise it is $\partition{2,1,\ldots,1}$ and $\ell_2=d-1$, but $\ell_1+\ell_2=d-2$.
If $\pi_2\supseteq\partition{1,1}$ we can switch the roles of $\pi_1$ and $\pi_2$ and use the first fact we noted to deduce that $\calD$ is realizable.
So we can assume $\pi_2=\partition{3,\ldots,3}*\rho_2$ with $\rho_2\subset \partition{1,2}$, and there are at least two $3$'s otherwise $d\leqslant 6$.
This implies that if $\pi_1\supseteq\partition{2,2}$ we are in case (i) with roles switched, and $\calD$ is realizable. Otherwise we also have
$\pi_1=\partition{3,\ldots,3}*\rho_1$ with $\rho_1\subset \partition{1,2}$, hence for $j=1,2$ we have
$\ell_j\leqslant (d-(1+2))/3+2=d/3+1$, but then
$$d-2=\ell_1+\ell_2\leqslant\frac23d+2\ \Rightarrow\ d\leqslant12.$$
Finally in subcase~(iii) we claim that $\pi_2\supseteq\partition{3,3}$, otherwise for $j=1,2$ we have
$\ell_j\geqslant (d-3)/2+1$ whence the contradiction
$$d-2=\ell_1+\ell_2\geqslant (d-3)+2=d-1.$$
Then, switching roles, we are in subcase~(i) or~(ii), so we conclude that $\calD$ is realizable.

\medskip

\noindent
\textsc{Case 5:}
$\pi_1=\partition{\matebold{x}}*\rho_1$ with $x=\max(\pi_1)\geqslant4$, $\pi_2=\partition{\mathbf{2},2,\ldots,2}$.
We apply move $T_2$ with $x_1=1$ and $x_2=x-3$, getting
$$(S,S,d-2,3;\partition{1,x-3}*\rho_1,\partition{2,\ldots,2},\partition{s-1,d-s-1}).$$
This can be one of the items (1)-(6) in many different ways, namely:
\begin{itemize}
\item[(3)] with $d-2=2k,\ s-1=k$ and
\begin{itemize}
\item[(a)] $x-3=2,\ \rho_1=\partition{2,\ldots,2,3}$
\item[(b)] $x-3=3,\ \rho_1=\partition{2,\ldots,2}$
\end{itemize}
\item[(4)] with $d-2=4k+2$, whence $k\geqslant4$, $s-1=2k+1$ and
\begin{itemize}
\item[(a)] $x-3=k+1,\ \rho_1=\partition{1,\ldots,1,k+2}$
\item[(b)] $x-3=k+2,\ \rho_1=\partition{1,\ldots,1,k+1}$
\end{itemize}
\item[(5)] with $d-2=4k$, whence $k\geqslant4$, $s-1=2k-1$ and $x-3=k+1,\ \rho_1=\partition{1,\ldots,1,k+1}$
\item[(6)] with $h=2$, $d-2=2k$, whence $k\geqslant8$, $s-1=2p$ and $x-3=k+1,\ \rho_1=\partition{1,\ldots,1}$.
\end{itemize}
Correspondingly we see that $\calD$ is
\begin{itemize}
\item[(3-a)] $(T,S,2k+2,3;\partition{2,\ldots,2,3,5},\partition{2,\ldots,2},\partition{k+1,k+1})$
\item[(3-b)] $(T,S,2k+2,3;\partition{2,\ldots,2,\mathbf{6}},\partition{2,\ldots,2},\partition{k+1,k+1})$
\item[(4-a)] $(T,S,4k+4,3;\partition{1,\ldots,1,k+2,\bm{k+4}},\partition{2,\ldots,2},\partition{2k+2,2k+2})$
\item[(4-b)] $(T,S,4k+4,3;\partition{1,\ldots,1,k+1,\bm{k+5}},\partition{2,\ldots,2},\partition{2k+2,2k+2})$
\item[(5)] $(T,S,4k+2,3;\partition{1,\ldots,1,k+1,\bm{k+4}},\partition{2,\ldots,2},\partition{2k,2k+2})$
\item[(6)] $(T,S,2k+2,3;\partition{1,\ldots,1,\bm{k+4}},\partition{2,\ldots,2},\partition{2p+1,2(k-p)+1})$.
\end{itemize}
Now (3-a) is precisely the exceptional item~(11) of the statement. In all the other cases we perform a reduction move $T_2$ at the highlighted entry of $\pi_1$,
always with $x_1=2$, getting a realizable candidate branch datum, namely one that cannot be one of items~(1) to~(6):
\begin{itemize}
\item[(3-b)] $(S,S,2k,3;\partition{2,\ldots,2},\partition{2,\ldots,2},\partition{k,k})$
\item[(4-a)] $(S,S,4k+2,3;\partition{1,\ldots,1,2,k,k+2},\partition{2,\ldots,2},\partition{2k+1,2k+1})$
\item[(4-b)] $(S,S,4k+2,3;\partition{1,\ldots,1,2,k+1,k+1},\partition{2,\ldots,2},\partition{2k+1,2k+1})$
\item[(5)] $(S,S,4k,3;\partition{1,\ldots,1,2,k,k+1},\partition{2,\ldots,2},\partition{2k-1,2k+1})$
\item[(6)] $(S,S,2k,3;\partition{1,\ldots,1,2,k},\partition{2,\ldots,2},\partition{2p,2(k-p)})$.
\end{itemize}

\medskip

\noindent
\textsc{Case 6:}
$\max(\pi_1)=3,\ \pi_2=\partition{2,\ldots,2}$.
If $d=2k$ we have $\ell_2=k$, whence $\ell_1=k-2$, which easily implies that $\pi_1=\partition{3,3,\mathbf{3},\mathbf{3}}*\rho_1$.
If $s=2$ or $s=d-2$ Proposition~\ref{xy22:prop} implies that $\calD$ is realizable. Otherwise we apply a move $T_3$ getting
$$(S,S,2k-4,3;\partition{3,3,1,1}*\rho_1,\partition{2,\ldots,2},\partition{s-2,d-s-2})$$
which is realizable unless it is item (5), but we are assuming $d>8$.

\medskip

To finish the proof we only must show that the above cases cover all possibilities up to switching $\pi_1$ and $\pi_2$.
In fact, since $\ell_1+\ell_2=d-2$, we cannot have $\max(\pi_1)\leqslant2$ and $\max(\pi_2)\leqslant2$, otherwise
$\ell_1+\ell_2\geqslant d/2+d/2=d$, which is absurd. So up to switching $\pi_1$ and $\pi_2$ we have $\max(\pi_1)\geqslant3$.
If $\max(\pi_1)\geqslant4$ we have the following mutually exclusive possibilities:
\begin{itemize}
\item[(i)] $\pi_2=\partition{2,\ldots,2}$;
\item[(ii)] $\pi_2\neq\partition{2,\ldots,2}$ and $\pi_2\ni 2$;
\item[(iii)] $\pi_2\not\ni2$ and $\max(\pi_2)=3$;
\item[(iv)] $\pi_2\not\ni2$ and $\max(\pi_2)\geqslant 4$;
\end{itemize}
If (i) holds we are in Case 5, if (ii) holds we are in Case 1, if (iii) holds we are in Case 3, while if (iv) holds we either have
$\pi_1\not\ni2$, and we are in Case 2, or $\pi_1\ni2$, but since $\max(\pi_2)\geqslant 4,\ \pi_1\neq\partition{2,\ldots,2}$ we are in Case 1
with roles switched. Having shown that our cases cover all possibilities with $\max(\pi_1)\geqslant4$, we can assume
$\max(\pi_1)=3$ and $\max(\pi_2)\leqslant 3$. Then either $\pi_2=\partition{2,\ldots,2}$, and we are in Case 6, or
$\pi_2\neq\partition{2,\ldots,2}$, and we are in Case 4.
\end{proof}

\paragraph{Large genus covering surface}
We now prove the following:

\begin{thm}\label{bigg:thm}
For $g\geqslant2$ there is no exceptional candidate branch datum of the form
$$\calD=(g\cdot T,S,d,3;\pi_1,\pi_2,\partition{s,d-s}).$$
\end{thm}

\begin{proof}
By induction on $g\geqslant1$ we prove that the only exceptional $\calD$ as in the statement are items~(7) to~(11) in Theorem~\ref{main:thm}.
The base step $g=1$ is Theorem~\ref{torus:thm}. Now we assume $g\geqslant2$ and we do the inductive step.
Again~\cite[Proposition~5.3]{EKS} implies that we can assume $2\leqslant s\leqslant d-2$.
Moreover again the computer-aided analysis of~\cite{Zh06} shows that for $d\leqslant 20$ we have no exceptions, so
we assume $d\geqslant 21$. We now analyse various cases showing that there always exists a reduction move $T_j:\calD\leadsto\calD'$
such that $\calD'$ is not item~(11) in Theorem~\ref{main:thm}, which is enough. Note that if $\ell_j=\ell(\pi_j)$ we have $\ell_1+\ell_2=d-2g$.

\medskip

\noindent
\textsc{Case 1:}
$\pi_1=\partition{\matebold{x}}*\rho_1$, $\pi_2=\partition{\mathbf{2}}*\rho_2$ with $x\ge4$.
We apply $T_2$ with $x_1=1$ and $x_2=x-3$, getting the desired
$$\calD'=((g-1)\cdot T,S,d-2,3;\partition{1,x-3}*\rho_1,\rho_2,\partition{s-1,d-s-1}).$$

\medskip

\noindent
\textsc{Case 2:}
$\pi_1=\partition{\matebold{x}}*\rho_1, \pi_2=\partition{\matebold{y}}*\rho_2, $ with $x\geqslant4, y\geqslant3$, and $\pi_2\not\ni2$. Then we can apply $T_4$ getting the desired
$$\calD'=((g-1)\cdot T,S,d-2,3;\partition{x-2}*\rho_1,\partition{y-2}*\rho_2,\partition{s-1,d-s-1}).$$

\medskip

\noindent
\textsc{Case 3:}
$\max(\pi_1)=3$ and $\max(\pi_2)\leqslant 3$. Here we further distinguish some situations:
\begin{itemize}
  \item[(a)] $\pi_1\supseteq\partition{1,1}$, hence $\pi_1=\partition{\mathbf{3},1,1}*\rho_1$. Then we apply $T_1$ getting the desired
  $$\calD'=((g-1)\cdot T,S,d,3;\partition{1,1,1,1,1}*\rho_1,\pi_2,\partition{s,d-s})$$
  \item[(b)] $\pi_1=\partition{\mathbf{3},\mathbf{3}}*\rho_1,\ \pi_2=\partition{2,2}*\rho_2$; for $s=2$ or $d-2$ we conclude that $\calD$ is
  realizable by Proposition~\ref{xy22:prop}, otherwise we apply $T_3$ getting the desired
  $$\calD'=((g-1)\cdot T,S,d-4,3;\partition{1,1}*\rho_1,\rho_2,\partition{s-2,d-s-2})$$
  \item[(c)] $\pi_1\not\supseteq\partition{2,2}$. If $\pi_1\supseteq\partition{1,1}$ we are in case (3-a), so we can assume $\pi_1\not\supseteq\partition{1,1}$. Then of course
  $\pi_1\supseteq\partition{3,3}$. Now if $\pi_2\supseteq\partition{2,2}$ we are in case (3-b), so we can assume $\pi_2\not\supseteq\partition{2,2}$. If $\pi_2\not\supseteq\partition{3,3}$ we have $\ell_2\geqslant d-3$, but
  $\ell_1\geqslant d/3$, which combined with $\ell_1+\ell_2=d-2g$ gives $d+6g\leqslant 9$, which we are excluding. So we can assume $\pi_2\supseteq\partition{3,3}$, and again by case (3-b), switching roles, we can also assume that
  $\pi_1\not\supseteq\partition{2,2}$. We then have that $\pi_j=\partition{3,\ldots,3}*\rho_j$ with $\rho_j\subseteq\partition{1,2}$ for $j=1,2$, whence $\calD$ is
  realizable by Proposition~\ref{dessins:special-families:prop}
  \item[(d)] $\pi_1\not\supseteq\partition{3,3}$. This implies that $\ell_1\geqslant (d-3)/2+1$, namely $\ell_1\geqslant (d-1)/2$. If also $\pi_2 \not\supseteq\partition{3,3}$ then
  $\ell_2\geqslant (d-1)/2$ as well, which contradicts $\ell_1+\ell_2=d-2g$. So $\pi_2\supseteq\partition{3,3}$, but we can assume $\pi_1\supseteq\partition{2,2}$ by case (3-c), so
  we are in case (3-b) with roles switched, and again we conclude that $\calD$ is realizable
  \item[(e)] If none of the above holds, in particular $\pi_1\supseteq\partition{2,2}$ and $\pi_2\not\supseteq\partition{2,2}$. Now if $\pi_2\supseteq\partition{3,3}$ we are in case (3-b)
  with roles switched, so we can assume $\pi_2\not\supseteq\partition{3,3}$, therefore $\ell_2\geqslant d-3$, but $\ell_1\geqslant d/3$, which as above is excluded.
\end{itemize}
We cannot have $\max(\pi_1)=\max(\pi_2)=2$ otherwise $\ell_1,\ell_2\geqslant d/2$, but $\ell_1+\ell_2=d-2g$. So either, up to switching,
$\max(\pi_1)\geqslant 4$ or $\max(\pi_1)=3$ and $\max(\pi_2)\leqslant 3$. The latter situation is Case 3. In the former either $\pi_2\ni2$, whence Case 1,
or $\pi_2\not\ni2$, but $\pi_2$ is non-trivial, so $\max(\pi_2)\geqslant3$, and we are in Case 2.
\end{proof}


\section{Realizability for more than three branching points}\label{morepoints:sec}

We begin by citing~\cite[Complement~5.6]{EKS}:

\begin{prop}\label{deg4:prop}
The only exceptional branch datum $$\calD=(g\cdot T,S,4,n;\pi_1,\ldots,\pi_n)$$ is item~(13) in Theorem~\ref{main:thm}.
\end{prop}

The next result eventually completes the proof of Theorem~\ref{main:thm}:

\begin{thm}\label{morepoints:thm}
The only exceptional data
$$\calD=(g\cdot T,S,d,n;\pi_1,\ldots,\pi_{n-1},\partition{s,d-s})$$
with $n\geqslant 4$ are items~(12) and~(13) in Theorem~\ref{main:thm}.
\end{thm}

\begin{proof}
Within this proof we use the notation of Section~\ref{monodromy:sec}.
We proceed by induction on $n$. The base step $n=4$ requires some work.
First of all a computer-aided analysis based on~\cite{Zh06}
(see the Appendix for more details)
implies that for $d\leqslant16$ the only exceptional $\calD$ as in the statement
are item~(12) and item~(13) with $n=4$, so we assume $d\geqslant17$. Set $v_j=v(\pi_j)$.

\medskip

\noindent
\textsc{Case 1:} $v_1+v_2<d$. We can then apply the reduction move $A_1$ at $\pi_1$ and $\pi_2$, taking $\theta_1$ and $\theta_2$ as given by Proposition~\ref{smallv:existence:prop}, getting
$$\calD\leadsto\calD'=(g\cdot T,S,d,3;\pi,\pi_3,\partition{s,d-s})$$
which is realizable unless it is item (11), namely
$$\calD'=(T,S,2k,3;\partition{2,\ldots,2},\partition{2,\ldots,2,3,5},\partition{k,k}),$$
so $g=1$, $d=2k$ and $s=k$, with $k>8$.
Moreover either
\begin{itemize}
\item[(I)] $\pi=\partition{2,\ldots,2},\ \pi_3=\partition{2,\ldots,2,3,5}$, so $v_1+v_2=k,\ v_3=k+2$, or
\item[(II)] $\pi=\partition{2,\ldots,2,3,5},\ \pi_3=\partition{2,\ldots,2}$, so $v_1+v_2=k+2,\ v_3=k$.
\end{itemize}
In case (I) we can suppose $v_1\leqslant k/2$, hence
$$k+2<1+k+2\leqslant v_1+v_3\leqslant 3k/2+2<2k=d$$
so we can apply Proposition~\ref{smallv:existence:prop} and the reduction move $A_1$ to $\calD$ at $\pi_1$ and $\pi_3$ getting an analogous $\calD'$ with $v(\pi)>k+2$,
which cannot be item~(11), so it is realizable.

In case (II) instead, noting that $v_1+v_2=k+2$ we see that the following cases cover all the possibilities up to switching indices:
\begin{itemize}
\item[(a)] $v_1<k$ and $v_1\neq 2$
\item[(b)] $v_1=2$ and $\pi_2\neq\partition{2,\ldots,2}$
\item[(c)] $v_1=2$ and $\pi_2=\partition{2,\ldots,2}$.
\end{itemize}
In case (a), since $v_3=k$ we have $v_1+v_3<2k=d$, so we can apply Proposition~\ref{smallv:existence:prop} and the reduction move $A_1$ to $\calD$ at $\pi_1$ and $\pi_3$, getting
$\calD'=(g\cdot T,S,2k,3;\pi,\pi_2,\partition{k,k})$
with $v(\pi)=v_1+k$, that cannot be $k$ or $k+2$, so $\calD'$ is realizable.
In case (b) we have $v_1+v_3=2+k<2k=d$, so again we can apply Proposition~\ref{smallv:existence:prop} and the reduction move $A_1$ to $\calD$ at $\pi_1$ and $\pi_3$, getting
$\calD'=(g\cdot T,S,2k,3;\pi,\pi_2,\partition{k,k})$ with $v_2=k$ but $\pi_2\neq\partition{2,\ldots,2}$, so $\calD'$ is realizable.
In case (c) note that $\pi_1=\partition{1,\ldots,1,2,2}$ or $\pi_1=\partition{1,\ldots,1,3}$.
Here we apply Proposition~\ref{A1:move:prop} in its full strength, namely choosing $\theta_1$ and $\theta_2$. In both cases we take
$\theta_2=(1,2)(3,4)\cdots(2k-1,2k)$, while $\theta_1=(1,3)(2,5)$ for $\pi_1=\partition{1,\ldots,1,2,2}$ and $\theta_1=(1,3,5)$ for $\pi_1=\partition{1,\ldots,1,3}$. Then $\theta_1\cdot\theta_2$ is
respectively
$$(1,5,6,2,3,4)(7,8)\cdots(2k-1,2k)\qquad (1,2,3,4,5,6)(7,8)\cdots(2k-1,2k)$$
whence
$$\calD\leadsto\calD'=(T,S,2k,3;\partition{6,2,\ldots,2},\partition{2,\ldots,2},\partition{k,k})$$
which is realizable.

\medskip

\noindent
\textsc{Case 2:} $v_i+v_j\geqslant d$ for all $1\leqslant i<j\leqslant 3$.
Without loss of generality we can assume $v_3\geqslant d/2$. Noting that
$$v_1+v_2\geqslant d\qquad v_3+v(\partition{s,d-s})\geqslant 1+d-2=d-1$$
we can apply move $A_2$ to $\calD$ at $\pi_1$ and $\pi_2$ (since $d\geqslant 17$ the assumptions of Proposition~\ref{A2:move:prop} are verified), getting
$$\calD\leadsto \calD'=(g'\cdot T,S,d,3;\pi,\pi_3,\partition{s,d-s})$$
with $v(\pi)\geqslant d-2$. We can now compute
\begin{eqnarray*}
g' & = & \frac12(v(\pi)+v_3+v(\partition{s,d-s}))-d+1\\
   & \geqslant & \frac12\left(d-2+\frac d2+d-2\right)-d+1=\frac d4-1\geqslant2,
\end{eqnarray*}
so $\calD'$ is realizable by Theorem~\ref{bigg:thm} and the base step $n=4$ is complete.

\medskip

For the induction step, we assume $n\geqslant5$ and we take $\calD$ as in the statement.
If $d=2$ the only candidate branch datum
$$\left(g\cdot T,S,2,2g+2;\partition{2},\ldots,\partition{2}\right)$$
is realizable by Theorem~\ref{length-1:thm}, so we can assume $d\geqslant 3$.
If $d=4$ the conclusion follows from Proposition~\ref{deg4:prop}, so we also assume $d\neq 4$. Now if up to
permutation we have $v_1+v_2\leqslant d-1$ we can apply a reduction move $A_1$ (also making use of Proposition~\ref{smallv:existence:prop}) to $\calD$ at $\pi_1$ and $\pi_2$, getting $\calD'$
with $n-1$ non-trivial partitions. Otherwise we have $v_1+v_2\geqslant d-1$ and $v_3+v_4\geqslant d-1$, so we can apply a reduction move $A_2$ to $\calD$ at $\pi_1$ and $\pi_2$
(since $d\geqslant3$ the assumptions of Proposition~\ref{A2:move:prop} are verified), getting again $\calD'$ with $n-1$ non-trivial partitions.
By induction, such a $\calD'$ can only be exceptional if it is item~(12) in Theorem~\ref{main:thm}, so $n-1=4$ and $d=8$.
But for $n=5$ and $d=8$ a computer-aided search based on~\cite{Zh06} implies that $\calD$ is not exceptional
(see the Appendix for more details).
\end{proof}

\section*{Appendix. Computational results}

In~\cite{Zh06} Zheng described an algorithm to compute all the exceptional 
candidate branch data of the form $(g\cdot T,S,d,n;\pi_1,\ldots,\pi_n)$ for fixed values of $n$ and $d$, and used it to produce a list of all the exceptional data with $d\leqslant 20$ and $n=3$.
We have implemented Zheng's algorithm and extended his computations to $d\leqslant 29$ (for $n=3$), confirming the prime-degree conjecture up to this level.
The source code of our implementation is publicly available 
at~\cite{BaroniCode}.

As an example of what our code can do, we have analyzed the 
realizability of all the candidate branch data of the form
$(S,S,d,3;\pi_1,\pi_2,\pi_3)$ with $d=11$ or $d=12$ and $\ell(\pi_3)=2$. The result was that there is no exceptional
candidate datum for $d=11$, and precisely those of 
Table~\ref{exceptional-sphere:table} for $d=12$.
These findings fall within the range already analyzed by
Zheng in~\cite{Zh06}, but they do not appear explicitly in 
his paper. They are also in perfect agreement with~\cite{Pako}.

\begin{table}[ht]
\center
\begin{tblr}{colspec={l}}
\hline
$(S,S,12,3;\partition{2,2,2,2,2,2},\partition{2,2,2,2,2,2},\partition{5,7})$\\
$(S,S,12,3;\partition{2,2,2,2,2,2},\partition{2,2,2,2,2,2},\partition{4,8})$\\
$(S,S,12,3;\partition{2,2,2,2,2,2},\partition{2,2,2,2,2,2},\partition{3,9})$\\
$(S,S,12,3;\partition{2,2,2,2,2,2},\partition{2,2,2,2,2,2},\partition{2,10})$\\
$(S,S,12,3;\partition{2,2,2,2,2,2},\partition{2,2,2,2,2,2},\partition{1,11})$\\
$(S,S,12,3;\partition{2,2,2,2,2,2},\partition{1,2,2,2,2,3},\partition{6,6})$\\
$(S,S,12,3;\partition{2,2,2,2,2,2},\partition{1,1,1,3,3,3},\partition{6,6})$\\
$(S,S,12,3;\partition{2,2,2,2,2,2},\partition{1,1,1,1,4,4},\partition{5,7})$\\
$(S,S,12,3;\partition{2,2,2,2,2,2},\partition{1,1,1,1,1,7},\partition{6,6})$\\
$(S,S,12,3;\partition{2,2,2,2,2,2},\partition{1,1,1,1,1,7},\partition{4,8})$\\
$(S,S,12,3;\partition{2,2,2,2,2,2},\partition{1,1,1,1,1,7},\partition{2,10})$\\
$(S,S,12,3;\partition{1,1,1,1,1,1,1,1,1,3},\partition{6,6},\partition{6,6})$\\
$(S,S,12,3;\partition{1,1,1,1,1,1,1,5},\partition{3,3,3,3},\partition{6,6})$\\
$(S,S,12,3;\partition{1,1,1,1,1,1,1,5},\partition{3,3,3,3},\partition{3,9})$\\
$(S,S,12,3;\partition{1,1,1,1,1,1,1,1,4},\partition{4,4,4},\partition{4,8})$\\
\hline
\end{tblr}\bigskip
\mycap{List of exceptional data $(S,S,d,3;\pi_1,\pi_2,\pi_3)$ with $d=11$ or $d=12$ and $\ell(\pi_3)=2$}
\label{exceptional-sphere:table}
\end{table}

\bigskip

Getting to the facts that were referred to above, 
we have used the same algorithm to establish the correctness of Theorem~\ref{main:thm} in the cases where $n=4$ with $d\leqslant 16$, and $n=5$ with $d\leqslant 8$.
Our computations show that the only exceptional candidate branch data $(g\cdot T,S,d,4;\pi_1,\ldots,\pi_4)$ with $d\leqslant 16$ and $\ell(\pi_4)=2$ are as follows:
\begin{itemize}
\item $\left(T,S,4,4;\partition{2,2},\partition{2,2},\partition{2,2},\partition{1,3}\right)$;
\item $\left(2\cdot T,S,8,4;\partition{2,2,2,2},\partition{2,2,2,2},\partition{2,2,2,2},\partition{3,5}\right)$.
\end{itemize}
As for those of the form 
$(g\cdot T,S,d,5;\pi_1,\ldots,\pi_5)$ with $d\leqslant 8$ and $\ell(\pi_5)=2$, the only exceptional one is:
\begin{itemize}
\item $\left(2\cdot T,S,4,5;\partition{2,2},\partition{2,2},\partition{2,2},\partition{2,2},\partition{1,3}\right)$.
\end{itemize}


\vspace{.5cm}
\vspace{.5cm}

\noindent
Mathematical Institute\\
University of Oxford\\
Andrew Wiles Building, Woodstock Road\\
OX2 6GG OXFORD -- United Kingdom\\
\texttt{filippo.baroni@maths.ox.ac.uk}

\vspace{.5cm}

\noindent
Dipartimento di Matematica\\
Universit\`a di Pisa\\
Largo Bruno Pontecorvo, 5\\
56127 PISA -- Italy\\
\texttt{petronio@dm.unipi.it}

\end{document}